\documentclass[11pt,twoside]{article}
\usepackage{latexsym}
\usepackage{amssymb,amsbsy,amsmath,amsfonts,amssymb,amscd}
\usepackage{graphicx,color}
\usepackage{float,url}
\usepackage{cancel}
\usepackage[colorlinks=true, citecolor =blue]{hyperref}
\newcommand  \ind[1]  {   {1\hspace{-1.2mm}{\rm I}}_{\{#1\} }    }
\newcommand {\wt}[1] {{\widetilde #1}}
\newcommand{\commentout}[1]{}
\newcommand{\R}{\mathbb{R}}

\newcommand{\tx}{\wt x}
\newcommand{\tz}{\wt z}

\newcommand {\e}  {\varepsilon}

\newcommand {\sg} {\sigma}
\newcommand {\vp} {\varphi}
\newcommand {\vr} {\varrho}

\newcommand {\Chi} {{\bf \raise 2pt \hbox{$\chi$}} }

\newcommand {\sgn} { {\rm sgn} }

\newcommand {\cJ} { {\mathcal J} }

\newcommand {\cP} { {\mathcal P} }
\newcommand {\cT} { {\mathcal T} }

\newcommand{\dd}{{\rm d}}
\newcommand{\intot}{\int_0^t}
\newcommand {\f}   {\frac}
\newcommand {\p}   {\partial}
\newcommand{\dis}{\displaystyle}

\newcommand{\beq}{\begin{equation}}
\newcommand{\eeq}{\end{equation}}
\newcommand{\bea} {\begin{array}{rl}}
\newcommand{\eea} {\end{array}}
\newcommand{\bepa}{\left\{ \begin{array}{l}}
\newcommand{\eepa} {\end{array}\right.}

\newcommand{\intt}{\int \hskip-8pt \int}
\newcommand{\inttt}{\int \hskip-8pt \int \hskip-8pt \int}
\newcommand{\intttt}{\int \hskip-8pt \int \hskip-8pt \int \hskip-8pt \int}

\newcommand{\vip}{\vskip0.15cm}

\newtheorem{theorem}{Theorem}
\newtheorem{lemma}[theorem]{Lemma}
\newtheorem{definition}[theorem]{Definition}

\newcommand{\qed}{{\hfill $\square$}}

\title{A non-expanding transport distance for some structured equations}

\author{
Nicolas Fournier\thanks{Sorbonne Universit\'{e}, CNRS, Laboratoire de Probabilit\'e, Statistique et Mod\'elisation, F-75005 Paris, France. Email: nicolas.fournier@sorbonne-universite.fr}
\and 
Beno\^\i t Perthame\thanks{Sorbonne Universit\'{e}, CNRS, Universit\'{e} de Paris, Inria, Laboratoire Jacques-Louis Lions, F-75005 Paris, France. 
Email: benoit.perthame@sorbonne-universite.fr. }
\thanks{B.P. has received funding from the European Research Council (ERC) under the European Union's Horizon 2020 research and innovation programme (grant agreement No 740623). }
}
\date{\today}

\begin{document}
\maketitle
\pagestyle{plain}
\pagenumbering{arabic}

\begin{abstract}
Structured equations are a standard modeling tool in mathematical biology. They are
integro-differential equations where the unknown depends on one or several variables, 
representing the state or phenotype of individuals. A large literature has been devoted to many 
aspects of these equations and in particular to the study of measure solutions.
Here we introduce a transport distance closely related to the Monge-Kantorovich distance,
which appears to be non-expanding for several (mainly linear) examples of structured equations.
\end{abstract} 

\noindent{\makebox[1in]\hrulefill}\newline

%
\section*{Introduction}

The subject of structured equations arises in several areas of biology and extends ordinary differential 
equations by including parameters chosen because they bring some influence on the population dynamics, 
see \cite{CushingBook, MD_LN68, P_birkhauser}. This leads to various integro-differential equations 
and partial differential equation (P.D.E.)
which also appear in many other areas as physics, communication science and industry.  Besides the interesting 
modeling issues, the questions which have been considered are about  existence of solutions, entropy properties and, mostly, long 
term convergence to steady states, with possibly exponential rate of convergence.  
Another question concerns  measure solutions, possibly after  renormalization \cite{BCG2020, DDGW2019}. 
Furthermore, in the context of a nonlinear neuroscience problem, convergence of a particle system has recently 
been proved using transport costs  with specific costs precisely adapted to the  coefficients~\cite{FL2016}.

\vip

The present papers aims at showing that a simple variant of the Monge-Kantorovich transport distance appears to be non-expanding along several structured equations. These include the renewal equation and a few other models listed below.

\vip

We work in a state space that we denote by  $\cJ$, which can be $[0,\infty)$, $[0,\infty)\times T$ 
with $T$ a discrete torus, 
$[0,\infty)^2$ or $[0,\infty)\times \R^d$ and we always use a cost function 
$\varrho:\cJ \times\cJ\mapsto [0,\infty) $ which satisfies
$\varrho (x,x)=0$ and $\varrho (x,y)=\varrho (y,x)>0$ for $x \neq y$ and will typically take the form $\varrho (y,x)= \min(|x-y|, a)$ for some $a>0$ chosen according to the equation at hand. We recall that 
for two probability measures $u_1, u_2 \in \cP(\cJ)$, the transport cost is defined as 
\beq\bepa
\cT_{\varrho}(u_1, u_2) = \dis \inf_{v \in {\mathcal H}(u_1,u_2)} \intt \varrho(x,y) v(\dd x,\dd y), \\[10pt]
{\mathcal H}(u_1,u_2)= \{v \in \cP(\cJ\times\cJ) \quad \hbox{with marginals} \quad u_1 \; \hbox{and}\; u_2\}.
\eepa
\label{eq:wasserstein} \eeq
When $\varrho$ is a distance on $\cJ$, $\cT_{\varrho}$ is a distance on $\cP(\cJ)$ and, with a slight 
abuse of language,  one refers
to the Monge-Kantorovich distance.
Recent accounts about the theory can be found in the books \cite{VTOT, AGSbook, Santambrogio}.

\vip
Our approach relies on the coupling method, see the survey paper \cite{FoPe2021}. 
The first example of use of the coupling method, to our knowledge, can be traced back to 
Dobrushin \cite{Dob}, where the nonlinear
Vlasov equation is derived as mean-field limit of a deterministic system of interacting particles,
making use of some transport cost.
No P.D.E. is written for the coupling in \cite{Dob}, because everything may be expressed in terms of 
characteristics. See \cite[Section 3]{GoMoPa} for a P.D.E. analogue to Dobrushin's argument. 
In the same spirit, the Euler equation is derived from a deterministic system of interacting vortices 
in Marchioro-Pulvirenti \cite[Section 5.3]{MaPu}, using also a coupling argument, see also \cite{Ha} 
for a result with the strong 
transport distance $d_\infty$.
\vip

The paper is organized as follows. We begin with the renewal equation in order to present in 
details the results and method. Building on this, we extend the method to a system of renewal equations, 
some space-age structured equation, the multi-time renewal equation, the growth-fragmentation equation,
and to an age-size coupled model. All these equations are linear. We complete our study with a model 
with sexual reproduction, which is quadratic, and generates new difficulties. Our setting is very 
general and does not use uniqueness of solutions, therefore we complete them with a technical appendix 
devoted to a uniqueness result by the Hilbert duality method when further regularity on the coefficients 
is assumed.

\section{The renewal equation}
\label{sec:RenewalEq}

Our first  example, also the simpler, is the general renewal equation. It allows to introduce the method and to explain the choice of cost within the setting of the equation
\beq \bepa
 \f{\p u_t(x) }{\p t} + \f{\p[ g(x) u_t(x)] }{\p x} +d(x) u_t(x) = b(x) N(t) , \qquad t\geq 0, \, x\geq 0,
\\[5pt]
\dis  u_t(x=0)=0, \qquad  N(t)= \int_0^\infty d(x) u_t(\dd x), \qquad t\geq 0.
\eepa 
\label{renewalEq}\eeq
When $b=\delta_0$, the Dirac mass at $0$, and $ g\equiv 1$, we find the classical renewal equation~\cite{feller}. 
The more general version at hand is motivated by various models proposed in mathematical neuroscience, 
\cite{FL2016, MW2018, PPS2010}.
We assume that
\beq \label{renewMKas1}
g, d \in C([0,\infty)), \quad b \in \cP([0, \infty)), \quad \hbox{g is non-increasing}, \quad 
g(0)\geq 0, \quad d\geq 0.
\eeq
We will further suppose that 
\beq 
\exists\; a>0  \quad \text{such that} \quad a \leq  \inf_{|x-y| \leq a} \frac{|x-y|\max(d(x),d(y))}{|d(x)-d(y)|} .
\label{renewMKas2}
\eeq
Observe that this last condition holds true with $a=\min(a_0,1)$ as soon as $a_0>0$, where
\[
a_0= \inf_{| x-y| \leq 1} \frac{|x-y|\max(d(x),d(y))}{|d(x)-d(y)|}.
\]
For example, $d(x)=\alpha+\beta x^p$ satisfies such a condition, provided $\alpha>0$, $\beta\geq 0$
and $p\geq 1$, as well as any Lipschitz and uniformly positive function.

\begin{theorem}
Assume \eqref{renewMKas1}-\eqref{renewMKas2}. We consider the cost function on $(0,\infty)\times (0,\infty)$ defined by
$$
\varrho(x,y) = \min( a, | x-y|).
$$ 
For any $u_0^1,u_0^2 \in \cP([0,\infty))$, there exists a pair of weak measure solutions
$(u_t^1)_{t\geq 0},(u_t^2)_{t\geq 0} \subset \cP([0,\infty))$ to \eqref{renewalEq} starting from $u_0^1$ and
$u_0^2$, i.e., such that for $i=1,2$ and all $t\geq 0$,
\begin{gather}
\int_0^t  \hskip-5pt  \int_0^\infty  \hskip-5pt  d(x)u_s^i(\dd x)\dd s<\infty \label{ww1}
\end{gather}
and for all $t\geq 0$, all $\varphi \in C^1_c([0,\infty))$,
\begin{gather}
\int_0^\infty  \hskip-5pt  \varphi(x)u_t^i(\dd x) = \int_0^\infty \!\varphi(x)u_0^i(\dd x) + \intot \hskip-5pt \int_0^\infty \!
\Big[g(x)\varphi'(x)+ d(x)\int_0^\infty(\varphi(z)-\varphi(x))b(\dd z)  \Big] u_s^i(\dd x)\dd s.  \label{ww2}
\end{gather}
Moreover,  for all $t\geq 0$, we have
$$
{\mathcal T}_\varrho (u^1_t,u^2_t) \leq {\mathcal T}_\varrho(u^1_0, u^2_0).
$$
\label{th:tc:bm}
\end{theorem}

Notice that, even if we did not mention it, it follows from \eqref{ww2} that $u\in C_{\rm w} ([0, \infty); \cP([0,\infty))$. 
Also, the generality of this statement  relies on the price that the solutions $u^1_t, \, u^2_t$ may 
depend on the choice of the pair $(u_0^1,u_0^2)$. However, when $g\in C^1_b$,  the solutions are unique 
in distributional sense as proved in the appendix, and then the result is more standard. 
The regularity of $g$ can certainly be lowered in view of the theory developed in  
\cite{bianchiniG2011,BJM2005,dpl}.
\vip

\begin{proof} We assume \eqref{renewMKas1}, fix $u^1_0,u^2_0 \in \cP([0,\infty))$ and consider any 
$v_0 \in {\mathcal H}(u^1_0,u^2_0)$.
There exists a family $(v_t)_{t\geq 0}$ of probability measures on $[0,\infty)^2$, starting from $v_0$, such that for all $t\geq 0$,
\begin{align}\label{aa1}
\intot \hskip-5pt\intt [d(x)+d(y)]v_s(\dd x,\dd y) \dd s <\infty,
\end{align}
and which weakly solves
\begin{align*}
 \f{\p v_t }{\p t} &+ \f{\p   [g(x) v_t]}{\p x} +   \f{\p  [g(y) v_t]}{\p y} + \max (d(x), d(y)) v_t  =  \dis b(x) \delta (x- y) \intt \min(d(x') ,d(y'))  \; v_t(\dd x',\dd y')
\\
&+b(x) \int \big( d(x')- d(y) \big)_+ \; v_t(\dd x',y)  
\dis+b(y) \int \big( d(y')-  d(x) \big)_+ \; v_t(x,\dd y') .
\end{align*}
This means that for all $t\geq 0$, all  $\vp \in C^1_c([0,\infty)^2)$,
\begin{align}
\intt \vp(x,y)v_t(\dd x,\dd y) =& \intt \vp(x,y)v_0(\dd x,\dd y)
+ \intot  \hskip-5pt \intt \Big[ g(x) \f{\p \vp(x,y)}{\p x}  +g(y) \f{\p   \vp(x,y) }{\p y} \Big] v_s(\dd x,\dd y) \dd s\notag\\
&+\intot  \hskip-5pt \intt  \hskip-8pt \int \big[\varphi(z,z)-\varphi(x,y) \big] \min(d(x),d(y))b(\dd z)
v_s(\dd x,\dd y)\dd s\notag\\
& +\intot  \hskip-5pt \intt  \hskip-8pt \int \big[\varphi(z,y)-\varphi(x,y) \big] (d(x)-d(y))_+ b(\dd z)
v_s(\dd x,\dd y)\dd s \notag\\
&+\intot  \hskip-5pt \intt  \hskip-8pt \int \big[\varphi(x,z)-\varphi(x,y) \big] (d(y)-d(x))_+b(\dd z)
v_s(\dd x,\dd y)\dd s.
\label{abcd}
\end{align}
Using \eqref{aa1} and then \eqref{abcd} with a function $\vp$ depending only on $x$, observing that
$v_0\in {\mathcal H}(u^1_0,u^2_0)$ and that
$$
\min(d(x),d(y))+(d(x)-d(y))_+=d(x),
$$
we deduce that the first marginal $u^1_t(\dd x)=\int_{y\in [0,\infty)} v_t(\dd x,\dd y)$ satisfies
\eqref{ww1}-\eqref{ww2}. Similarly, the second marginal $u^2_t(\dd y)=\int_{x\in [0,\infty)} v_t(\dd x,\dd y)$ 
satisfies \eqref{ww1}-\eqref{ww2}. 
And it holds that $v_t \in {\mathcal H}(u^1_t,u^2_t)$ for all $t\geq 0$.
\vip
The existence for \eqref{aa1}-\eqref{abcd} follows from classical arguments, 
using e.g., an approximate problem where $g,d$ are replaced by smooth and bounded functions, 
and from the following {\it a priori} tightness estimate.
By the de la Vall\'ee Poussin theorem, there exists a function $h:[0,\infty)\to [0,\infty)$ such that
$\lim_{x\to \infty} h(x)=\infty$ and such that
$$
C:=\intt [h(x)+h(y)] v_0(\dd x,\dd y) + \int h(z) b(\dd z) <\infty.
$$
One can moreover choose $h$ smooth and satisfying $0 \leq h'\leq 1$. Applying \eqref{abcd} with 
$\vp(x,y)=h(x)+h(y)$, one immediately  concludes that for all $t\geq 0$,
\begin{align*}
\intt [h(x)+h(y)]v_t(\dd x,\dd y) =& \intt [h(x)+h(y)]v_0(\dd x,\dd y)
+ \intot  \hskip-5pt  \intt [ g(x)h'(x)+g(y)h'(y)]v_s(\dd x,\dd y) \dd s \\
&\hskip-1.5cm+ \intot  \hskip-5pt  \intt  \Big[d(x) \Big(\int h(z) b(\dd z) - h(x)\Big) + d(y) \Big(\int h(z) b(\dd z) - h(y)\Big) \Big]v_s(\dd x,\dd y) \dd s
\\
\leq & C + 2 \overline C t + \intot  \hskip-5pt  \intt  \Big[d(x) (C-h(x))+ d(y) (C-h(y))\Big]v_s(\dd x,\dd y) \dd s,
\end{align*}
where $\overline C = \sup_{x \geq 0} g(x) h'(x)$ is finite because $g$ is continuous and non-increasing and 
because $h'$ is $[0,1]$-valued. Since now $d$ is continuous, 
non-negative and since $h$ increases to infinity, there is $L>0$
such that $d(x) (C-h(x)) \leq L -d(x)h(x)/2$, whence finally
$$
\intt [h(x)+h(y)]v_t(\dd x,\dd y) + \frac12\intot  \hskip-5pt  \intt \big[d(x)h(x)+d(y)h(y)\big] v_s(\dd x,\dd y) \dd s \leq 
C+2 \overline C t +2Lt.
$$
Since $\lim_{x\to \infty}h(x)=\infty$, this last {\it a priori} tightness estimate is sufficient 
to prove existence for \eqref{aa1}-\eqref{abcd}.

\vip

We next fix $a>0$, set $\rho(x,y)=\min(|x-y|,a)$ and we choose a coupling $v_0 \in {\mathcal H}(u^1_0,u^2_0)$ such that
$\int \hskip-5pt \int \rho(x,y)v_0(\dd x,\dd y)=\cT_\rho(u^1_0,u^2_0)$. We apply~\eqref{abcd} with
$\vp=\rho$ (or  more precisely,  firstly to some smooth and compactly supported approximation $\rho_\e$ of $\rho$
and then let $\e\to0$). We notice that
$$
g(x) \f{\p \rho(x,y)}{\p x}  +g(y) \f{\p   \rho(x,y) }{\p y} = \ind{|x-y|\leq a}\sgn(x-y)[g(x)-g(y)] \leq 0
$$
because $g$ is non-increasing. Since $b$ is a probability measure, we obtain
\begin{align*}
\intt \rho(x,y)v_t(\dd x,\dd y) \leq & \cT_\rho(u^1_0,u^2_0) - \intot  \hskip-5pt  \intt \rho(x,y)\max(d(x),d(y)) 
v_s(\dd x,\dd y)\dd s\\
&+   \intot  \hskip-5pt  \intt \hskip-8pt \int \Big[\rho(z,y)(d(x)-d(y))_+ + \rho(x,z)(d(y)-d(x))_+\Big] b(\dd z)
v_s(\dd x,\dd y)\dd s.
\end{align*}
Recalling that $\cT_\rho(u^1_t,u^2_t) \leq \int \hskip-5pt \int \rho(x,y)v_t(\dd x,\dd y)$, in order to complete the proof,  it is therefore sufficient to verify the inequality,  for all $x, \, y \geq 0$,
\begin{equation}\label{ttt}
I(x,y):= \f{\big( d(x)-d(y)\big)_+ }{ \max (d(x), d(y))} \int \varrho(z, y) b(\dd z) 
+ \f{\big( d(y)-d(x) \big)_+ }{ \max (d(x), d(y))}\int  \varrho(x, z)  b(\dd z)
\leq \varrho(x,y).
\end{equation}
Since $\varrho(x, y) = \min(a, | x-y|)$ and since $b$ is a probability measure, we have
$$
I(x,y) \leq \f{( d(x)-d(y)\big)_+ }{ \max (d(x), d(y))} a 
+ \f{\big( d(y)-d(x) \big)_+ }{ \max (d(x), d(y))} a = \frac{ |d(x)-d(y)|}{\max (d(x), d(y))} a 
\leq \min(|x-y|,a).
$$
Indeed, the last inequality is obvious if $|x-y|\geq a$ and follows from \eqref{renewMKas2} otherwise.
\end{proof}\qed

\section{A system of renewal equations}
\label{sec:sysRen}

In mathematical biology, it is usual to describe the cell cycle, see \cite{MD_LN68, P_birkhauser}, using a system of renewal equations coupled at there boundaries as follows 
\beq \bepa\label{sys}
 \f{\p u_t (x,i)}{\p t} + \f{\p [g_i(x) u_t(x,i)] }{\p x} +d_i(x) u_t(x,i) =0, 
\qquad t \geq 0, \; x \geq 0, \; i=1,..., I,
\\[5pt]
\dis u_t(x=0,i) = \int_0^\infty d_{i-1}(x) u_t(\dd x,i-1) , \qquad\quad t \geq 0, \; i=1,..., I,
\eepa \eeq
with the convention $u_t(x,0)=u_t(x,I)$. 
In terms of stochastic processes, a cell  with age $x$ in state $i$ ages chronologically (when, say,
$g_i\equiv 1$) until, with rate $d_i(x)$, it changes state to $i+1$ where it starts with age $x=0$. 
So the state space is 
\[
\cJ =[0,\infty)\times T, \; \text{with $T=\{1,...,I\}$ the torus, i.e., states $0$ and $I+1$ are identified 
to states $I$ and $1$}.
\]
We assume that for all $i\in T$,
\beq \label{AsSy1}
g_i, \, d_i \in C([0,\infty)), \qquad \hbox{$g_i$ is non-increasing}, \qquad g_i(0)\geq 0, \quad d_i\geq 0,
\eeq
and that 
\beq 
\exists \;  a>0 \qquad \text{such that} \qquad   a \leq \min_{1\leq i \leq I} \inf_{|x-y| \leq a} \frac{|x-y|\max(d_i(x),d_i(y))}{|d_i(x)-d_i(y)|} .
\label{renewAs4}
\eeq
This condition is satisfied if for all $i\in T$, there are $\alpha_i>0$, $\beta_i\geq 0$
and $p_i\geq 1$ such that $d_i(x)=\alpha_i+\beta_i x^{p_i}$,
or if all $d_i$'s are Lipschitz and uniformly positive.

\begin{theorem} Assume \eqref{AsSy1}-\eqref{renewAs4}. 
We consider the cost on $\cJ\times \cJ$ defined by
\[
\varrho(x,i,y,j) = \min( a, | x-y|) \ind{ i = j} + a \ind{i \neq j}.
\]
For any $u_0^1,u_0^2 \in \cP(\cJ)$,
there exists a pair of weak measure solutions 
$(u_t^1)_{t\geq 0},(u_t^2)_{t\geq 0} \subset \cP(\cJ)$ to \eqref{sys} starting from $u_0^1$ and
$u_0^2$, i.e., such that for $j=1,2$ and all $t\geq 0$
\begin{gather}
\int_0^t \! \int_\cJ \! d_i(x)u_s^j(\dd x, \dd i)\dd s<\infty
\label{sww1}
\end{gather}
and  for $j=1,2$, all $t\geq 0$ and all $\varphi \in C^1_c(\cJ)$,
\begin{align}
\int_\cJ \varphi(x,i)u_t^j(\dd x,\dd i) =& \int_\cJ \!\varphi(x,i)u_0^j(\dd x,\dd i)\notag\\ 
&+ \intot  \!\int_\cJ \!
\Big[g_i(x)\varphi'(x,i)+ d_i(x)[\varphi(0,i+1)-\varphi(x,i)] \Big] u_s^j(\dd x,\dd i)\dd s .\label{sww2}
\end{align}
Moreover,  for all $t\geq 0$, we have
$$
{\mathcal T}_\varrho (u^1_t,u^2_t) \leq {\mathcal T}_\varrho(u^1_0, u^2_0).
$$
\end{theorem}
\vip
Note that if $u_t(x,i)$ is a strong solution to \eqref{sys}, then $u_t(\dd x,\dd i):=\sum_{k\in T} u_t(x,k)\dd x
\delta_k(\dd i)$ is a weak measure solution to \eqref{sys}. Note also that the comments after 
Theorem~\ref{th:tc:bm} apply here too.
\vip

\begin{proof}
We assume \eqref{AsSy1}-\eqref{renewAs4} and consider any coupling $v_0 \in {\mathcal H}(u^1_0,u^2_0)$.
There exists a family $(v_t)_{t\geq 0}$ of probability measures on $\cJ^2$ such that for all $t\geq 0$,
\begin{align}\label{rr1}
\int_0^t  \hskip-5pt\intt [d_i(x)+d_j(y)] v_s(\dd x,\dd i,\dd y, \dd j) \dd s<\infty,
\end{align}
and which weakly solves
\begin{align*}
 \f{\p v_t }{\p t} + &\f{\p  [g_i(x) v_t] }{\p x} +   \f{\p [g_j(y) v_t]}{\p y} + \big[(d_i(x) \vee d_j(y)) \ind{ i = j}  + (d_i(x) +d_j(y) ) \ind{i \neq j} \big] v_t = 
\\
& \quad \delta_{(0,0)}(x,y)  \ind{ i = j} \intt  (d_{i-1}(x') \wedge d_{i-1}(y'))  \; v_t (\dd x', i-1, \dd y',j-1)  
\\
&+\delta_{0}(x)  \ind{ j= i-1}  \int \big( d_{i-1}(x')- d_{j}(y) \big)_+ \; v_t (\dd x',i-1, y,j) 
\\
&+\delta_{0}(y)  \ind{ i = j-1}  \int \big( d_{j-1}(y') -d_{i}(x') \big)_+ \; v_t (x,i, \dd y', j-1 ) 
\\
&+\delta_{0}(x)  \ind{ j \neq i-1}  \hspace{-4pt} \int  d_{i-1}(x') v_t ( \dd x',i-1,y,j) 
+\delta_{0}(y)  \ind{ i \neq j-1}   \hspace{-4pt} \int  d_{j-1}(y') v_t ( x,i,\dd y', j-1) . 
\end{align*}
This means that for all $\vp \in C^1_c(\cJ^2)$, all $t\geq 0$,
\begin{align}
\intt \vp(x,i,y,j)v_t(&\dd x,\dd i,\dd y, \dd j)=
\intt \vp(x,i,y,j)v_0(\dd x,\dd i,\dd y, \dd j) \notag\\
&+ \intot \hskip-5pt \intt \Big[  g_i(x)\f{\p  \vp(x,i,y,j)}{\p x} +  g_j(y)\f{\p \vp(x,i,y,j)}{\p y}\Big] 
v_s(\dd x,\dd i,\dd y, \dd j) \dd s \notag\\
&+ \intot \hskip-5pt \intt \Big(
\big[ \vp (0,i+1,0,j+1) - \vp (x,i,y,j)\big] \ind{ i = j}\big( d_{i}(x)\land d_j(y)\big)
\notag\\
& \hspace{20mm}+ \big[ \vp (0,i+1,y,j) - \vp (x,i,y,j)\big] \ind{ i = j} \big( d_{i}(x)- d_{j}(y) \big)_+
\notag \\
& \hspace{20mm}+ \big[ \vp (x,i,0,j+1) - \vp (x,i,y,j)\big] \ind{ i = j} \big( d_{j}(y)- d_{i}(x) \big)_+
\notag \\
& \hspace{20mm}+ \big[ \vp (0,i+1,y,j) - \vp (x,i,y,j)\big] \ind{ i \neq j} d_{i}(x)
\notag \\
&\hspace{20mm} + \big[ \vp (x,i,0,j+1) - \vp (x,i,y,j)\big] \ind{ i \neq j} d_{j}(y)  \Big)
v_s(\dd x,\dd i,\dd y, \dd j) \dd s. \label{rr2}
\end{align}
Using \eqref{rr1} and \eqref{rr2} with a function $\vp$ depending only on $(x,i)$ (or $(y,j)$),
we see that the first marginal $u^1_t(\dd x,\dd i)= \int_{(y,j)\in \cJ} v_t(\dd x,\dd i,\dd y, \dd j)$
(or the second marginal $u^2_t(\dd y,\dd j)= \int_{(x,i)\in \cJ} v_t(\dd x,\dd i,\dd y, \dd j)$)
satisfies \eqref{sww1} and \eqref{sww2}.

\vip
As before, the existence for \eqref{rr1}-\eqref{rr2} follows from classical arguments, 
using an approximate problem where $g_i,d_i$ are replaced by smooth and bounded functions, 
and from the following {\it a priori} tightness estimate.
By the de la Vall\'ee Poussin theorem, there exists a function $h:[0,\infty)\to [0,\infty)$ such that
$\lim_{x\to \infty} h(x)=\infty$ and such that
$$
C:=\intt [h(x)+h(y)] v_0(\dd x,\dd i,\dd y,\dd j) <\infty.
$$
One can moreover choose $h$ smooth, satisfying $h(0)=0$ and $0 \leq h'\leq 1$. Applying \eqref{rr2} with 
$\vp(x,i,y,j)=h(x)+h(y)$, one immediately  concludes that for all $t\geq 0$,
\begin{align*}
\intt [h(x)+h(y)]v_t(\dd x,\dd i,\dd y, \dd j) =& \intt [h(x)+h(y)]v_0(\dd x,\dd i,\dd y, \dd j)\\
&+ \intot  \hskip-5pt \intt [ g_i(x)h'(x)+g_j(y)h'(y)]v_s(\dd x,\dd i,\dd y, \dd j) \dd s\\
&- \intot  \hskip-5pt \intt [d_i(x)h(x)+d_j(y)h(y)]v_s(\dd x,\dd i,\dd y,\dd j) \dd s.
\end{align*}
Setting $\overline C = \sup_{i\in T,x \geq 0} g_i(x) h'(x)$, which is finite because $g_i$ is continuous, non-increasing and 
because $h'$ is $[0,1]$-valued, we end up with
$$
\intt [h(x)+h(y)]v_t(\dd x,\dd i,\dd y,\dd j) + \intot \hskip-5pt \intt \big[d_i(x)h(x)+d_j(y)h(y) \big] v_s(\dd x,\dd i,\dd y,\dd j)
\dd s \leq C+2 \overline C t.
$$
This last {\it a priori} tightness estimate is sufficient 
to prove existence for \eqref{rr1}-\eqref{rr2}.

\vip

We now apply \eqref{rr2} with $\vp=\varrho$, after regularization as before, 
where we recall that $\varrho(x,i,y,j) = \min( a, | x-y|) \ind{ i = j} + a \ind{i \neq j}$.
Since $g_i$ is non-increasing, we have
$$
g_i(x)\f{\p  \varrho(x,i,y,j)}{\p x} +  g_j(y)\f{\p \varrho(x,i,y,j)}{\p y}
=\ind{i=j}\ind{|x-y|\leq a} \sgn(x-y)[g_i(x)-g_j(y)]\leq 0.
$$
Hence we find, choosing $v_0$ such that $\int\hskip-5pt\int \varrho(x,i,y,j)v_0(\dd x,\dd i,\dd y, \dd j)
=\cT_\rho(u^1_0,u^2_0)$,
\begin{align*}
\intt \varrho(x,i,y,j)&v_t(\dd x,\dd i,\dd y, \dd j)\\
 \leq \cT_\rho(u^1_0,u^2_0) &+ \intot  \hskip-5pt \intt \Big(
\big[ \vr (0,i+1,0,j+1) - \vr (x,i,y,j)\big] \ind{ i = j}\big( d_{i}(x)\land d_j(y)\big)
\notag\\
& \hspace{20mm}+ \big[ \vr (0,i+1,y,j) - \vr (x,i,y,j)\big] \ind{ i = j} \big( d_{i}(x)- d_{j}(y) \big)_+
\notag \\
& \hspace{20mm}+ \big[ \vr (x,i,0,j+1) - \vr (x,i,y,j)\big] \ind{ i = j} \big( d_{j}(y)- d_{i}(x) \big)_+
\notag \\
& \hspace{20mm}+ \big[ \vr (0,i+1,y,j) - \vr (x,i,y,j)\big] \ind{ i \neq j} d_{i}(x)
\notag \\
&\hspace{20mm} + \big[ \vr (x,i,0,j+1) - \vr (x,i,y,j)\big] \ind{ i \neq j} d_{j}(y)  \Big)
v_s(\dd x,\dd i,\dd y, \dd j) \dd s.
\end{align*}
With our choice of $\vr$, the two last lines are non-positive. We thus arrive at 
\begin{align*}
\intt \varrho(x,i,y,j) v_t(\dd x,\dd i,\dd y, \dd j)
\leq \cT_\rho(u^1_0,u^2_0) 
-\intot  \hskip-5pt \intt \ind{i=j}\Delta_i(x,y) v_s(\dd x,\dd i,\dd y, \dd j) \dd s,
\end{align*}
where
\begin{align*}
\Delta_i(x,y)=& (d_i(x) \vee d_i(y))(|x-y|\land a) - ( d_{i}(x)- d_{i}(y))_+ a - 
( d_{i}(y)- d_{i}(x))_+ a \\
=&(d_i(x) \vee d_i(y))(|x-y|\land a) - |d_{i}(x)- d_{i}(y)| a. 
\end{align*}
Next, we  check that $\Delta_i$ is always non-negative.
If $|x-y|\geq a$, this is obvious. If $|x-y|\leq a$, this follows from
\eqref{renewAs4}. Since $\cT_\rho(u^1_t,u^2_t)\leq \int\hskip-5pt\int \rho(x,i,y,j)v_t(\dd x,\dd i,\dd y,\dd j)$, 
the proof is complete.
\end{proof} \qed

\section{Space and age structure}

Next,  we consider an example similar to that in the previous section, when the discrete 
parameter $i$ is replaced by a continuous parameter $z\in \R^d$, which represents space or a physiological trait. The formalism makes the link with the heat equation through a standard physical process used in  particular to describe diffusion or anomalous diffusion,  see recent analyses in \cite{NPT2018, CGM2019, BerryLG2016}. 
We depart from the equation 
\beq \bepa
\dis \e^2 \f{\p u_t(x,z) }{\p t}  + \f{\p u_t(x,z) }{\p x}  +d(x)  u_t(x,z)=0, \qquad\; t \geq 0, \, x >0, \, z \in \R^d,
\\ [10pt]
\dis u_t(x=0, z)= \int_{0}^\infty  \hskip-8pt  \int_{\R^d} d(x)  u_t(\dd x, z+ \e \eta) k(\dd \eta), 
\qquad t \geq 0, \, z \in \R^d.
\eepa
\label{eq.renewSpace}
\eeq
This equation models particles characterized by their age $x$ and position $z$. 
When in state $x,z$, the particle's age $x$ grows linearly until there is a jump,
at rate $d(x)$, resulting in the particle moving  from $z$ to $z-\e \eta$ (with $\eta$ chosen according to the probability
density $k$). At each jump, the age is reset to $0$.
Hence the state space is here $\cJ=[0,\infty)\times\R^d$. 
Under a few assumptions, it is known that, for $u_\e$ the solution to
\eqref{eq.renewSpace}, $U_\e(t,z)=\int_0^\infty u_\e(t,x,z)\dd x$
converges, as $\e\to 0$, to the solution of the heat equation in $\R^d$.
\vip

We assume that 
\begin{align}\label{aas1}
d \in C([0,\infty)),\quad d\geq 0,\quad k \in \cP(\R^d),
\end{align}
and, again, that
\beq
\exists \;  a>0 \qquad \text{such that } \quad a\leq \inf_{|x-y]\leq a} \f{ |x-y|\max (d(x), d(y)) } { |d(x)-d(y)|} .
\label{aas2}
\eeq

\begin{theorem} Assume \eqref{aas1}-\eqref{aas2} and fix $\e>0$.
We consider the cost on $\cJ\times \cJ$ defined by
\[
\varrho(x,z,y,r) = \min( a, | x-y| + |z-r|).
\]
For any $u_0^1,u_0^2 \in \cP(\cJ)$,
there exists a pair of weak measure solutions
$(u_t^1)_{t\geq 0},(u_t^2)_{t\geq 0} \subset \cP(\cJ)$ to \eqref{eq.renewSpace} starting from $u_0^1$ and
$u_0^2$, i.e., such that for $i=1,2$ and all $t\geq 0$,
\begin{gather}
\int_0^t \hskip-5pt \int_\cJ \! d(x)u_s^i(\dd x, \dd z)\dd s<\infty  \label{aasww1}
\end{gather}
and for  $i=1,2$, all $t\geq 0$ and all $\varphi \in C^1_c(\cJ)$,
\begin{align}
\int_\cJ \varphi(x,z)u_t^i&(\dd x,\dd z) = \int_\cJ \!\varphi(x,z)u_0^i(\dd x,\dd z)\notag\\ 
&+ \e^{-2}\intot  \hskip-5pt \int_\cJ \!
\Big[\frac{\p \varphi(x,z)}{\p x}+ d(x) \int_{\R^d}[\varphi(0,z-\e \eta)-\vp(x,z)]k(\dd \eta) \Big] 
u_s^i(\dd x,\dd z)\dd s. \label{aasww2}
\end{align}
Moreover, for all $t\geq 0$, we have 
$$
{\mathcal T}_\varrho (u^1_t,u^2_t) \leq {\mathcal T}_\varrho(u^1_0, u^2_0).
$$
\label{th:tcRenSA}
\end{theorem}

\bigskip

\begin{proof} We assume \eqref{aas1}-\eqref{aas2}, consider $u^1_0,u^2_0 \in \cP(\cJ)$
and a coupling $v_0 \in {\mathcal H}(u^1_0,u^2_0)$. There exists a family $(v_t)_{t\geq 0}$ of probability
measures on $\cJ^2$ such that for all $t\geq 0$ 
\begin{align}\label{ss1}
\int_0^t \intt [d(x)+d(y)] v_t(\dd x,\dd z,\dd y, \dd r)<\infty,
\end{align}
and which weakly solve
\begin{align*} 
\e^2 \f{\p v_t}{\p t}    + \f{\p v_t}{\p x} +& \f{\p v_t}{\p y}  +(b(x)\vee b(y))   v_t =
\\
& \quad \delta (x) \delta(y)  \int_{0}^\infty \hskip-8pt   \int_{0}^\infty \hskip-8pt \int_{\R^d} (b(x') \wedge b(y')) v_t(\dd x', z+\e \eta, \dd y', r+\e \eta) k(\eta) \dd \eta
\\
& +  \delta (x)   \int_{\cJ} \big(b(x') - b(y)\big)_+  v_t(\dd x', z+\e \eta,y, r + \e \eta) k(\eta) \dd \eta
\\
& +  \delta (y)   \int_{\cJ} \big(b(y') - b(x)\big)_+  v_t(x, z+ \e \eta, \dd y', r +\e \eta) k(\eta) \dd \eta .
\end{align*} 
This means that for all $\vp \in C^1_c(\cJ^2)$, all $t\geq 0$,
\begin{align}
\intt& \vp(x,z,y,r)v_t(\dd x,\dd z,\dd y, \dd r)=
\intt \vp(x,z,y,r)v_0(\dd x,\dd z,\dd y, \dd r) \notag\\
&+ \e^{-2}\intot \hskip-5pt  \intt \Big[\f{\p  \vp(x,z,y,r)}{\p x} +  \f{\p \vp(x,z,y,r)}{\p y}\Big] 
v_s(\dd x,\dd z,\dd y, \dd r) \dd s \notag\\
&+ \e^{-2}\intot \hskip-5pt \intt \Big(
\big[ (d(x)\land d(y)) \int_{\R^d} [\vp (0,z-\e \eta,0,r-\e \eta) - \vp(x,z,y,r)] k(\dd \eta) \notag\\
& \hspace{20mm}+  (d(x)-d(y))_+ \int_{\R^d} [\vp (0,z-\e \eta,y,r) - \vp(x,z,y,r)] k(\dd \eta)\notag \\
& \hspace{20mm}+  (d(y)-d(x))_+ \int_{\R^d} [\vp (x,z,0,r-\e\eta) - \vp(x,z,y,r)] k(\dd \eta)
\Big) v_s(\dd x,\dd z,\dd y, \dd r) \dd s. \label{ss2}
\end{align}
Using \eqref{ss1} and \eqref{ss2} with a function $\vp$ depending only on $(x,z)$ (or $(y,r)$),
we see that the first marginal $u^1_t(\dd x,\dd z)= \int_{(y,r)\in \cJ} v_t(\dd x,\dd z,\dd y, \dd r)$
(and the second one $u^2_t(\dd y,\dd r)= \int_{(x,z)\in \cJ} v_t(\dd x,\dd z,\dd y, \dd r)$)
satisfies \eqref{aasww1} and \eqref{aasww2}.

\vip
The existence for \eqref{rr1}-\eqref{rr2} follows as usual from the following {\it a priori} tightness estimate.
By the de la Vall\'ee Poussin theorem, there is a function $h:[0,\infty)\to [0,\infty)$ such that
$\lim_{x\to \infty} h(x)=\infty$ and
$$
C:=\intt [h(x)+h(y)+h(|z|)+h(|r|)] v_0(\dd x,\dd z,\dd y,\dd r)  + \int h(\e |\eta|) k(\dd \eta) <\infty.
$$
One can moreover choose $h$ smooth, satisfying $h(0)=0$ and $0 \leq h'\leq 1$. Applying \eqref{ss2} with 
$\vp(x,z,y,r)=h(x)+h(y)+h(|z|)+h(|r|)$, one immediately  concludes that for all $t\geq 0$,
\begin{align*}
\intt [h(x)&+h(y)+h(|z|)+h(|r|)]v_t(\dd x,\dd z,\dd y, \dd r ) \leq C
+ \e^{-2}\intot \hskip-5pt \intt [h'(x)+h'(y)]v_s(\dd x,\dd z,\dd y, \dd r) \dd s\\
&+ \e^{-2}\intot \hskip-5pt \intt \Big[d(x)\int [h(0)+h(|z-\e \eta|)-h(x)-h(|z|)]k(\dd \eta)\\
&\hskip3cm+d(y)\int [h(0)+h(|r-\e \eta|)-h(y)-h(|r|)]k(\dd \eta) \Big]v_s(\dd x,\dd z,\dd y, \dd r) \dd s.
\end{align*}
Using that $h'$ is $[0,1]$-valued and that $h(0)=0$, we observe that
$$
\int \big[h(0)+h(|z-\e \eta|)-h(x)-h(|z|) \big]k(\dd \eta)\leq \int h(|\e \eta|)(\dd \eta)-h(x) \leq C-h(x).
$$
Using now that $d(x)[C-h(x)] \leq L - d(x)h(x)/2$ for some constant $L>0$, we end up with
\begin{align*}
\intt [h(x)\!+h(y)\!+h(|z|)\!+h(|r|)]v_t(\dd x,\dd z,\dd y, \dd r) 
+\frac1{2\e^2} \intot \hskip-5pt \intt &[d(x)h(x)+d(y)h(y)]v_s(\dd x,\dd z,\dd y, \dd r) \dd s\\
&\leq C
+ 2 \e^{-2} t + 2 \e^{-2}L t.
\end{align*}
This {\it a priori} tightness estimate is sufficient, as usual, 
to prove existence for \eqref{ss1}-\eqref{ss2}.
\vip

We now apply \eqref{ss2} with $\vp=\varrho$, where we recall that
$\varrho(x,z,y,r)= \min( a, | x-y| + |z-r|)$. Since
$$
\f{\p  \varrho(x,z,y,r)}{\p x} +  \f{\p \varrho(x,z,y,r)}{\p y}=0,
$$
we find, choosing $v_0$ such that $\int\hskip-5pt\int \varrho(x,z,y,r)v_0(\dd x,\dd z,\dd y, \dd r)
=\cT_\rho(u^1_0,u^2_0)$,
\begin{align*}
\intt \varrho&(x,z,y,r)v_t(\dd x,\dd z,\dd y, \dd r)\leq \cT_\rho(u^1_0,u^2_0)
+ \e^{-2}\intot \hskip-5pt \intt \Delta(x,z,y,r)v_s(\dd x,\dd z,\dd y, \dd r) \dd s,
\end{align*}
where
\begin{align*}
 \Delta(x,z,y,r)=&(d(x)\land d(y)) \int_{\R^d} [\vr (0,z-\e \eta,0,r-\e \eta) - \vr(x,z,y,r)] k(\dd \eta) \notag\\
& +(d(x)-d(y))_+ \int_{\R^d} [\vr (0,z-\e \eta,y,r) - \vr(x,z,y,r)] k(\dd \eta)\notag \\
& +(d(y)-d(x))_+ \int_{\R^d} [\vr (x,z,0,r-\e\eta) - \vr(x,z,y,r)] k(\dd \eta).
\end{align*}
Since $\cT_\rho(u^1_t,u^2_t)\leq \int\hskip-5pt\int \rho(x,z,y,r)v_t(\dd x,\dd z,\dd y,\dd r)$, 
it thus only remains to check that $\Delta$ is always non-positive. But we have, assuming
e.g. that $d(x)\geq d(y)$,
\begin{align*}
\Delta(x,z,y,r)=& -d(x)[(|x-y|+|z-r|)\land a] + d(y) (|z-r|\land a)\\
&+(d(x)-d(y))\int_{\R^d} [(|y|+|z-\e\eta-r|)\land a]k(\dd\eta)\\
\leq & -d(x)[(|x-y|+|z-r|)\land a] + d(y) (|z-r|\land a)+(d(x)-d(y))a.
\end{align*}
If $|x-y|+|z-r|\geq a$, we have
$$
\Delta(x,z,y,r) \leq -d(x)a+d(y)a+(d(x)-d(y))a= 0.
$$
If $|x-y|+|z-r|\leq a$, still assuming that $d(x)\geq d(y)$, we have
\begin{align*}
\Delta(x,z,y,r) \leq& -d(x)(|x-y|+|z-r|)+d(y)|z-r|+(d(x)-d(y))a \\
\leq & -d(x)|x-y|+(d(x)-d(y))a \leq 0
\end{align*}
thanks to \eqref{aas2}, since  $|x-y|+|z-r|\leq a$ implies that $|x-y|\leq a$.
Therefore, we always  have  $\Delta (x,z,y,r) \leq 0$ and the proof is complete.
\end{proof} \qed

\section{The multiple time renewal equation}
\label{sec:severaltimes}

Several applications use multi-time renewal equations to describe a population density subjected to aging 
or to time-evolution. Recently, for  evaluating the efficiency of tracing softwares,  
it was used to take into account secondary infections, see~\cite{ferretti}. In neuroscience, the 
interpretation is that neurones keep memory of their last spikes in the process of deciding when to fire 
again, see~\cite{CCDR}. For two times memory, the equation reads

\beq \bepa
\f{\p u_t(x_1,x_2)}{\p t} + \f{\p u_t(x_1,x_2)}{\p x_1}  + \f{\p u_t(x_1,x_2)}{\p x_2}   +d(x_1,x_2) u_t(x_1,x_2)=0,  \qquad x_2 \geq  x_1 \geq 0,\; t\geq 0 ,
\\ [10pt]
u_t(x_1=0, x_2)= \int_{x_1}^\infty d(x_2,z) u_t(x_2,\dd z),\hskip4.2cm x_2\geq 0,\; t\geq 0.
\eepa
\label{eq.multirenewal}
\eeq
An example of  stochastic  interpretation is that  particles are individuals producing {\it events} at a rate depending on the ages of its two last
events. Here the variables $x_1$ and $x_2$ represent the ages of the two last events, so these ages
increase linearly until, at rate $d(x_1,x_2)$, they are reset to the values $(0,x_1)$.
Hence our state space is now $\cJ = \{(x_1,x_2) \in [0,\infty)^2 : x_2 > x_1\}$.
We assume that
\begin{align}\label{da1}
d \in C(\cJ),\quad d\geq 0
\end{align}
and that 
\begin{align}\label{da2}
\exists \; a>0 \quad \hbox{such that}\quad a \leq \inf_{ 2 |x_1-\wt x_1 |+|x_2- \wt x_2|\leq a} \f{ ( | x_1- \wt x_1| +|x_2- \tilde x_2| )  \max (d(x_1, x_2), d( \wt x_1, \wt x_2))} { |d(x_1, x_2)-d(\wt x_1,  \wt x_2)|} .
\end{align}
One can check that $d(x_1,x_2)= \alpha + \beta x_1^{p_1}+ \gamma  x_2^{p_2}$, with
$\alpha>0$, $\beta\geq 0$, $\gamma\geq 0$, $p_1\geq 1$ and $p_2\geq 1$ satisfies such an assumption,
as well as Lipschitz and uniformly positive functions.

\begin{theorem} We assume \eqref{da1}-\eqref{da2}.
We consider the cost on $\cJ\times \cJ$ defined by
\[
\varrho(x_1,x_2,\tx_1,\tx_2) = \big[2|x_1-\wt x_1| +  |x_2-\wt x_2| \big]\wedge a.
\]
For any $u_0^1,u_0^2 \in \cP(\cJ)$,
there exists a pair of weak measure solutions
$(u_t^1)_{t\geq 0},(u_t^2)_{t\geq 0} \subset \cP(\cJ)$ to \eqref{eq.multirenewal} starting from $u_0^1$ and
$u_0^2$, i.e., such that for $i=1,2$ and  all $t\geq 0$
\begin{gather}
\int_0^t  \hskip-5pt  \int_\cJ \! d(x_1,x_2)u_s^i(\dd x_1,\dd x_2)\dd s<\infty  \label{dww1}
\end{gather}
and for  $i=1,2$, all $t\geq 0$ and all $\varphi \in C^1_c(\cJ)$,
\begin{align}
\int_\cJ \varphi(x_1&,x_2)u_t^i(\dd x_1,\dd x_2) = \int_\cJ \!\varphi(x_1,x_2)u_0^i(\dd x_1,\dd x_2)\notag\\ 
&+\intot \hskip-5pt \int_\cJ \!\Big[\frac{\p \varphi(x_1,x_2)}{\p x_1}+ \frac{\p \varphi(x_1,x_2)}{\p x_2}
+d(x_1,x_2) [\varphi(0,x_1)-\vp(x_1,x_2)] \Big] 
u_s^i(\dd x_1,\dd x_2)\dd s. \label{dww2}
\end{align}
Moreover, for all $t\geq 0$, we have
$$
{\mathcal T}_\varrho (u^1_t,u^2_t) \leq {\mathcal T}_\varrho(u^1_0, u^2_0).
$$
\label{th:tcMultiTimet}
\end{theorem}

\begin{proof} We assume \eqref{da1}-\eqref{da2}, consider $u^1_0,u^2_0 \in \cP(\cJ)$
and a coupling $v_0 \in {\mathcal H}(u^1_0,u^2_0)$. There exists a family $(v_t)_{t\geq 0}$ of probability
measures on $\cJ^2$ such that for all $t\geq 0$ 
\begin{align}\label{dss1}
\int_0^t  \hskip-5pt  \intt [d(x_1,x_2)+d(\tx_1,\tx_2)] v_t(\dd x_1,\dd x_2,\dd \tx_1, \dd \tx_2)<\infty
\end{align}
and which weakly solves, with zero flux boudary conditions at $x_1=0$ and $x_2=0$,
\begin{align*}
\f{\p v_t}{\p t}   + \f{\p v_t}{\p x_1}  + \f{\p v_t}{\p x_2} &+ \f{\p v_t}{\p  \wt  x_1}  + \f{\p v_t}{\p  \wt  x_2}   + (d(x_1,x_2)\vee d(\wt x_1, \wt x_2))  v_t
\\[10pt]
&=\delta(x_1)  \delta( \wt x_1)   \intt   \big(d(x_2,z)\vee d(\wt x_2, \wt z) \big)  v_t(x_2, \dd z,  \wt x_2, \dd \wt z)  
\\
&\quad + \delta(x_1)    \intt   \big( d (x_2,z) -d(\wt x_1, \wt x_2) \big)_+  v_t(x_2,\dd z, \wt x_1, \wt x_2)
\\
&\quad + \delta( \wt x_1)    \intt   \big( d(x_1,x_2)- d(\wt x_2, \wt z) \big)_+ v_t(x_1, x_2, \wt x_2, \dd \wt z).
\end{align*}
This means  that for all $\vp \in C^1_c(\cJ^2)$, all $t\geq 0$,
\begin{align}
&\intt \vp(x_1,x_2,\tx_1,\tx_2)v_t( \dd x_1,\dd x_2,\dd \tx_1, \dd \tx_2)=
\intt \vp(x_1,x_2,\tx_1,\tx_2)v_0(\dd x_1,\dd x_2,\dd \tx_1, \dd \tx_2)\notag\\
&\hskip0.5cm+ \intot \hskip-5pt \intt \Big[\f{\p \vp}{\p x_1} + \f{\p \vp}{\p x_2}  + \f{\p \vp}{\p \tx_1} + \f{\p \vp}{\p \tx_2}
\Big](x_1,x_2,\tx_1,\tx_2)v_s(\dd x,\dd z,\dd y, \dd r) \dd s \notag\\
&\hskip0.5cm+\intot \hskip-5pt \intt \Big((d(x_1,x_2)\land d(\tx_1,\tx_2))[\vp(0,x_1,0,\tx_1)-\vp(x_1,x_2,\tx_1,\tx_2)]\notag\\
&\hskip2cm + (d(x_1,x_2)-d(\tx_1,\tx_2))_+[\vp(0,x_1,\tx_1,\tx_2)-\vp(x_1,x_2,\tx_1,\tx_2)]\label{dss2}\\
&\hskip2cm + (d(\tx_1,\tx_2)-d(x_1,x_2))_+[\vp(x_1,x_2,0,\tx_1)-\vp(x_1,x_2,\tx_1,\tx_2)]
\Big) v_s(\dd x_1,\dd x_2,\dd \tx_2, \dd \tx_2) \dd s. \notag 
\end{align}
Using \eqref{dss1} and \eqref{dss2} with a function $\vp$ depending only on $(x_1,x_2)$ (or $(\tx_1,\tx_2)$),
we see that the marginals of $v$ satisfy \eqref{dww1} and \eqref{dww2}.

\vip
The existence for \eqref{dss1}-\eqref{dss2} follows as usual from the following {\it a priori} tightness estimate.
By the de la Vall\'ee Poussin theorem, there is a function $h:[0,\infty)\to [0,\infty)$ such that
$\lim_{x\to \infty} h(x)=\infty$ and
$$
C:=\intt [h(x_1)+h(x_2)+h(\tx_1)+h(\tx_2)] v_0(\dd x_1,\dd x_2,\dd \tx_1,\dd \tx_2)<\infty.
$$
Choosing moreover $h$ smooth, satisfying $h(0)=0$ and $0 \leq h'\leq 1$, applying \eqref{dss2} with 
$ \vp(x_1,x_2,\tx_1,\tx_2)= h(x_1)+h(x_2)+h(\tx_1)+h(\tx_2)$, one easily concludes as usual
that for all $t\geq 0$,
\begin{align*}
&\intt [h(x_1)+h(x_2)+h(\tx_1)+h(\tx_2)] \, v_t(\dd x_1,\dd x_2,\dd \tx_1,\dd \tx_2)\\
&\hskip0.5cm+ \intot \hskip-5pt \intt \big(d(x_1,x_2) h(x_2)+d(\tx_1,\tx_2) h(\tx_2) \big)
v_s(\dd x_1,\dd x_2,\dd \tx_1,\dd \tx_2)\dd s \leq C + 4t.
\end{align*}

Recalling that $h$ increases to infinity and that $h(x_1)\leq h(x_2)$ for all $(x_1,x_2)\in \cJ$,
this {\it a priori} tightness estimate is enough to prove existence for \eqref{dss1}-\eqref{dss2}.

\vip
We now apply \eqref{ss2} with $\vp=\varrho$, where
$\varrho(x_1,x_2,\tx_1,\tx_2)= \min( a, 2| x_1-\tx_1|+|x_2-\tx_2|)$. Since
$$
\f{\p \vp}{\p x_1} + \f{\p \vp}{\p x_2}  + \f{\p \vp}{\p \tx_1} + \f{\p \vp}{\p \tx_2}=0,
$$
we find, choosing $v_0$ such that $\int\hskip-5pt\int \varrho(x_1,x_2,\tx_1,\tx_2)
v_0(\dd x_1,\dd x_2,\dd \tx_1, \dd \tx_2)=\cT_\rho(u^1_0,u^2_0)$,
\begin{align*}
\intt \vr(x_1,x_2,\tx_1,\tx_2)v_t( \dd x_1,\dd x_2,\dd \tx_1, \dd \tx_2)
\leq \cT_\rho(u^1_0,u^2_0)
+\! \intot \hskip-5pt \intt \Delta(x_1,x_2,\tx_1,\tx_2)v_s(\dd x_1,\dd x_2,\dd \tx_1, \dd \tx_2) \dd s,
\end{align*}
where
\begin{align*}
\Delta(x_1,x_2,\tx_1,\tx_2)=&-(d(x_1,x_2)\vee d(\wt x_1, \wt x_2))([ 2 |x_1-\wt x_1| + |x_2-\wt x_2 | ] \wedge a)
\\
&+(d(x_1,x_2)\land d(\wt x_1, \wt x_2)) ( |x_1-\tx_1|\land a)\\
&+ \big( d( x_1,  x_2) -d(\wt x_1, \wt x_2) \big)_+   ([2 |\wt x_1| +| x_1- \wt x_2|] \wedge a )
\\
&+ \big( d( \tx_1,  \tx_2) -d(x_1, x_2) \big)_+   ([2 |x_1| +| \tx_1- x_2|] \wedge a).
\end{align*}
Since $\cT_\rho(u^1_t,u^2_t)\leq \int\hskip-5pt\int 
\rho(x_1,x_2,\tx_1,\tx_2)v_t( \dd x_1,\dd x_2,\dd \tx_1, \dd \tx_2)$, 
it only remains to check that $\Delta$ is always non-positive. 
\vip
When first $2 |x_1-\wt x_1| + |x_2-\wt x_2 |\geq a$, this is obvious.
\vip
When $2 |x_1-\wt x_1| + |x_2-\wt x_2 |\leq a$, it suffices to verify that
\begin{align*}
(d(x_1,x_2)\vee d(\wt x_1, \wt x_2))&[ 2 |x_1-\wt x_1| + |x_2-\wt x_2 | ]\\
\geq& (d(x_1,x_2)\land d(\wt x_1, \wt x_2)) |x_1-\tx_1| + |d( x_1,  x_2) -d(\wt x_1, \wt x_2)| a. 
\end{align*}
This follows from the fact that
$$
(d(x_1,x_2)\vee d(\wt x_1, \wt x_2))[|x_1-\wt x_1| + |x_2-\wt x_2 | ]
\geq |d( x_1,  x_2) -d(\wt x_1, \wt x_2)| a
$$
by \eqref{da2}.
\end{proof} \qed

\section{Growth-fragmentation}
\label{sec:GF}

The growth-fragmentation equation arises in several areas of biology. The variable represents for
instance the size of cells or the length  of biopolymers. It also arises in
communication science for TCP connections. A large literature is available  on the subject and we refer 
for instance to \cite{MD_LN68, DDGW2019, Monmarche2015, BertoinW2020}.
The model combines growth with a rate $g$ and fragmentation with a rate $d$ and it is written
\begin{equation} \left\{
\begin{array}{l}
\dis \f{\p u_t(x)}{\p t} + \f{\p [g(x)\, u_t(x)]}{\p x} +d(x) \, u_t(x) =  
\int_x^\infty  d(x') \kappa(x,x')  \, u_t(\dd x') , \qquad t\geq 0, \; x\geq 0,
\\[8pt]
u_t(x=0) =  0, \hskip8.9cm t\geq 0.
\end{array} \right.
\label{eq:sizestructN}
\end{equation}
Usual conditions on the fragmentation kernel are expressed through the identities 
\begin{equation}  
\kappa(x,x') = 0  \; \text{ for } x>x', \qquad \quad   \int_0^{x'} \kappa(x,x') \dd x = 1 
\label{eq:ssNas1}
\end{equation}
which lead to the conservation law
\[ 
\int_0^\infty u_t(x) \dd x = \int_0^\infty u_0 (x) \dd x=1.
\]
To go further, we can specify
\beq
\kappa(x,x')  = \f 1 {x'} \,  \beta\Big(\f x {x'}\Big) \quad \hbox{for some} \quad \beta \in \cP([0,1]),
\label{tc:gfkernel}
\eeq
see the end of the section for a more general possible setting. We then assume that
\begin{equation} 
g, \, d \in C([0, \infty)), \quad g \text{ is non-increasing}, \quad g(0) \geq 0,  \quad d\geq 0,
\label{eq:sizestructAsA}
\end{equation}
and
\beq
\exists \; a >0 \quad \text{such that} \quad  a \leq \left(1 -  \int_0^1 r \beta(\dd r) \right)  
\inf_{|x-y]\leq a} \f{ |x-y|\max (d(x), d(y)) } { |d(x)-d(y)|}.
\label{tc:gfas1}
\eeq
If $\beta$ is non-trivial in that $\int_0^1 r \beta(\dd r)\in [0,1)$, then this assumption
is verified if e.g. $d(x)=\alpha+\beta x^p$, provided $\alpha>0$, $\beta\geq 0$ and $p\geq 1$.

 %
\begin{theorem} We assume \eqref{tc:gfkernel}-\eqref{eq:sizestructAsA}-\eqref{tc:gfas1}.
We choose again, on $[0,\infty) \times [0,\infty)$ the cost
$$
\varrho(x,y) = \min( a, | x-y|) .
$$
For any $u_0^1,u_0^2 \in \cP([0,\infty))$, there exists a pair of weak measure solutions
$(u_t^1)_{t\geq 0},(u_t^2)_{t\geq 0} \subset \cP([0,\infty))$ to \eqref{eq:sizestructN}, starting from $u_0^1$ and
$u_0^2$, i.e., such that for $i=1,2$, for all $t\geq 0$, all $A\geq 1$,
\begin{gather}
\intot \int_{A+1}^\infty d(x) \int_0^{A/x}\beta(\dd r) u_s^i(\dd x) \dd s  <\infty \label{gfww1}
\end{gather}
and for all $t\geq 0$, all $\varphi \in C^1_c([0,\infty))$,
\begin{gather}
\int_0^\infty  \hskip-5pt  \varphi(x)u_t^i(\dd x) \!= \!\int_0^\infty \!\varphi(x)u_0^i(\dd x) +\! \intot \hskip-5pt \int_0^\infty \!
\Big[g(x)\varphi'(x)+ d(x)\! \int_0^1( \varphi(r x)-\varphi(x))\beta(\dd r) \Big] u_s^i(\dd x)\dd s.  
\label{gfww2}
\end{gather}
Moreover, for all $t\geq 0$, we have 
$$
{\mathcal T}_\varrho (u^1_t, u^2_t) \leq {\mathcal T}_\varrho(u^1_0, u^2_0).
$$
\label{th:tcgf}
\end{theorem}

Observe that \eqref{gfww2} makes sense thanks to \eqref{gfww1}:
for $\varphi \in C^1_c([0,\infty))$ supported in $[0,A]$, 
$$
\Big|d(x)\! \int_0^1( \varphi(r x)-\varphi(x))\beta(\dd r)\Big| \leq \ind{x\leq A+1} 2\|\vp\|_\infty \sup_{[0,A+1]} d 
+ \ind{x>A+1} \| \vp \|_\infty d(x)\int_0^{A/x} \beta(\dd r),
$$
where we used a rough upper bound by $A+1$ to fit our assumption.
\\

\begin{proof} 
Consider $u^1_0,u^2_0 \in \cP([0,\infty))$
and a coupling $v_0 \in {\mathcal H}(u^1_0,u^2_0)$. There exists a family $(v_t)_{t\geq 0}$ of probability
measures on $[0,\infty)^2$ satisfying, for all $t\geq 0$, all $A\geq 1$,
\begin{gather}\label{etac}
\intot \intt  \Big[\ind{x\geq A+1}d(x) \int_0^{A/x}\beta(\dd r)+\ind{y\geq A+1}d(y) 
\int_0^{A/y}\beta(\dd r) \Big] v_s(\dd x,\dd y) \dd s <\infty
\end{gather}
and solving weakly
\begin{align*}
\dis \f{\p v_t}{\p t} &+ \f{\p}{\p x} \big[g(x)\, v_t\big] + \f{\p}{\p y} \big[g(y)\, v_t\big] +
(d(x)\vee d(y)) \,v_t 
= 
\\&  \intt  (d(x')\wedge d(y'))  \f{1}{x'y'} \beta \big( \f{x}{x'}  \big) \delta \big( \f{x}{x'}- \f{y}{y'}   \big)
 \beta \big( \f{x}{x'}  \big)   \, v_t(\dd x',\dd y')  
\\
&+ \int  (d(x') -d(y))_+ \f{1}{x'} \beta \big( \f{x}{x'}  \big)  v_t(\dd x',y)
  + \int  (d(y') - d(x))_+  \f{1}{y'} \beta \big( \f{y}{y'}  \big)  v(x,\dd y').
\end{align*}
This means that for all $t\geq 0$, all $\varphi \in C^1_c([0,\infty)^2)$,
\begin{align}
\intt& \vp(x,y)  v_t(\dd x,\dd y) =  \intt \!\varphi(x,y)v_0(\dd x, \dd y)  + \int_0^t  \hskip-5pt \intt   \Big[g(x)  \f{\p \vp(x,y) }{\p x} + g(y) \f{\p \vp(x,y) }{\p y} \Big] v_s(\dd x,\dd y) \dd s\notag \\ 
& +  \intot \intt  (d(x)\wedge d(y))    \int_0^1  \big[ \vp(r x,r y) -  
\vp(x,y) \big]\beta(\dd r) v_s(\dd x,\dd y) \dd s \notag \\[5pt]
& + \intot \intt  (d(x)- d(y))_+  \int_0^1 \big[ \vp(rx,y) - \vp(x,y) \big] \beta(\dd r)v_s(\dd x,\dd y)  \dd s
\notag  \\[5pt]
& + \intot \intt (d(y)- d(x))_+  \int_0^1  \big[ \vp(x,ry) - \vp(x,y) \big]\beta(\dd r)v_s(\dd x,\dd y) \dd s.
\label{etac2}
\end{align}
One checks as usual that the two marginals of $v_t$ solve \eqref{gfww1}-\eqref{gfww2}. Next, as in the previous sections, 
one finds  that there is a function $h:[0,\infty)\to[0,\infty)$,
strictly increasing to infinity, which taken as a test function gives
\begin{align*}
\intt &[h(x)+h(y)]v_t(\dd x,\dd y)
+ \intot \intt \big[ d(x) H(x) + d(\tx) H(\tx) \big] v_s(\dd x,\dd y) \dd s   \leq C + 2\overline{C} t,
\end{align*}
where $H(x)= \int_0^1 (h(x)-h(rx))\beta(\dd r)$ and 
$\overline{C}= \sup_{x \geq 0} g(x) h'(x)$. This {\it a priori} estimate is sufficient to prove tightness and construct a solution to 
\eqref{gfww1}-\eqref{gfww2}, because for any $A\geq 1$, any $x\geq A+1$, one has
$$
\int_0^{A/x} \beta(\dd r) \leq \int_0^{A/x} \frac{h(x)-h(rx)}
{h(A+1)-h(A)} \beta(\dd r) \leq \frac{H(x)}{h(A+1)-h(A)}.
$$

Applying now the above equation to $\varphi=\vr$, where $\varrho(x,y) = \min( a, | x-y|)$, one finds as usual,
if choosing $v_0$ correctly and using that $g$ is non-increasing, that
$$
\intt \vr(x,y)  v_t(\dd x,\dd y) \leq \cT_\vr(u^1_0,u^2_0) + \intot \intt \Delta(x,y)v_s(\dd x,\dd y) \dd s,
$$
where
\begin{align}
\Delta(x,y)=&-(d(x) \vee d(y)) \varrho(x,y) 
+ (d(x) \wedge d(y)) \int_0^1  \varrho(rx,ry) \beta(\dd r)\notag
\\[8pt]
&+(d(x)- d(y))_+  \int_0^1 \varrho(rx,y)\beta(\dd r)+(d(y)- d(x))_+  \int_0^1 \varrho(x,ry)\beta(\dd r).
\label{tc:gfcond}
\end{align}
We finally show that $\Delta$ is alway non-positive. If first $|x-y|\geq a$, it is enough to
check that
$$
a (d(x)\lor d(y)) \geq a(d(x)\land d(y)) + a (d(x)- d(y))_++ a (d(y)- d(x))_+,
$$
which is obvious.  If next $|x-y|\leq a$, it suffices to check that
$$
(d(x)\lor d(y))|x-y| \geq (d(x)\land d(y)) \Big[\int_0^1 r\beta(\dd r)\Big]|x-y| + 
a (d(x)- d(y))_++ a (d(y)- d(x))_+,
$$
which follows from the fact that
$$
\Big[1-\int_0^1 r\beta(\dd r)\Big](d(x)\lor d(y))|x-y| \geq a |d(x)- d(y)|
$$
thanks to \eqref{tc:gfas1}.
\end{proof} \qed

\vip

A little study shows that the result still holds true if assuming, instead of \eqref{tc:gfkernel},
that the family of fragmentation kernels $(\kappa(\cdot,x))_{x\in \R_+} \subset \cP(\R_+)$ satisfies
$\kappa([0,x],x)=1$ for all $x\geq 0$ and 
$$
\exists \; m\in [0,1), \quad \hbox{for all $x,y\in\R_+$,} \quad W_1(\kappa(\cdot,x),\kappa(\cdot,y)) \leq m|x-y|,
$$
where $W_1$ is the usual Monge-Kantorovich distance on $\cP(\R_+)$, and
replacing \eqref{tc:gfas1} by
$$
\exists \; a >0 \quad \text{such that} \quad  a \leq (1 - m)
\inf_{|x-y]\leq a} \f{ |x-y|\max (d(x), d(y)) } { |d(x)-d(y)|}.
$$
Indeed, it suffices to apply the usual strategy, starting from the coupling equation :
\begin{align*}
\intt \vp(x,y)  &v_t(\dd x,\dd y) =  \intt \!\varphi(x,y)v_0(\dd x, \dd y)  + \int_0^t  \hskip-5pt \intt   \Big[g(x)  \f{\p \vp(x,y) }{\p x} + g(y) \f{\p \vp(x,y) }{\p y} \Big] v_s(\dd x,\dd y) \dd s \\ 
& +  \intot \intt  (d(x)\wedge d(y))    \int_0^x\!\!\int_0^y  \big[ \vp(x',y') -  
\vp(x,y) \big]\overline \kappa(\dd x', \dd y', x,y) v_s(\dd x,\dd y) \dd s \\[5pt]
& + \intot \intt  (d(x)- d(y))_+  \int_0^x \big[ \vp(x',y) - \vp(x,y) \big] \kappa(\dd x',x)v_s(\dd x,\dd y)  \dd s
 \\[5pt]
& + \intot \intt (d(y)- d(x))_+  \int_0^y  \big[ \vp(x,y') - \vp(x,y) \big]\kappa(\dd y',y)v_s(\dd x,\dd y) \dd s,
\end{align*}
where for each $x,y\in \R_+$, $\overline \kappa (\cdot,\cdot,x,y) \in {\mathcal H}(\kappa(\cdot, x),\kappa(\cdot,y))$ 
satisfies
$$\intt |x'-y'|\overline \kappa(\dd x', \dd y', x,y)=
W_1(\kappa(\cdot,x),\kappa(\cdot,y)).
$$

\section{Age and size structure}

Models have been proposed which use several stucture variables. For instance, 
age or size only are not enough to predict cell division.
But a combination of both (or other physiological variables as size 
increment) have been used, see~\cite{DHKR, DOR2020}, leading to write
\begin{equation} \left\{
\begin{array}{ll}
\dis \f{\p u_t(x,z) }{\p t}  + \f{\p u_t(x,z)}{\p x}+ \f{\p[g(z)\, u_t(x,z)]}{\p z} +
d(x,z) \, u_t(x,z) =  0, \quad & t\geq 0,\; x>0,\; z>0,
\\[8pt]
u_t(x, z=0)=0, & t\geq 0,\; x>0,\\[5pt]
\dis u_t(x=0, z) =  \int_{x=0}^\infty   \int_{z'=z}^\infty d(x', z') \kappa(z,z')
u_t(\dd x' , \dd z'), & t\geq 0,\; z>0
\end{array} \right.
\label{eq:AgeSizestruct}
\end{equation}
The state space is $ \cJ=[0,\infty)^2$. We assume that the coagulation kernel has
the specific form \eqref{tc:gfkernel}, that
\begin{equation}\label{rra}
g \in C([0,\infty)), \quad g \text{ is non-increasing}, \quad g(0) \geq 0,  \quad 
d \in C(\cJ), \quad d\geq 0.
\end{equation}
and that
\beq
\exists \; a>0 \quad \hbox{such that} \quad a \leq \Big(1 -    \int_0^1 r \beta(\dd r) \Big) 
 \inf_{|x-z|+| \wt x- \wt z|\leq a} \f{  | x- \wt x| +|z- \tilde z|  } { |d(x,z)-d( \wt x,  \wt z)|} 
\max (d(x,z), d( \wt x, \wt z)).
 \label{tc:ASstructas1Z}
 \eeq

\begin{theorem} Assume \eqref{tc:gfkernel}-\eqref{rra}-\eqref{tc:ASstructas1Z} and consider the cost
$$
\varrho(x,z, \wt x, \wt z) = \min( a, | x- \wt x| +|z- \tilde z| ).
$$
For any $u_0^1,u_0^2 \in \cP(\cJ)$, there exists a pair of weak measure solutions
$(u_t^1)_{t\geq 0},(u_t^2)_{t\geq 0} \subset \cP(\cJ)$ to \eqref{eq:AgeSizestruct}, starting from $u_0^1$ and
$u_0^2$, i.e., such that for $i=1,2$ and all $t\geq 0$, all $A\geq 1$,
\begin{gather}
\intot \int_\cJ  d(x,z)\Big[ \ind{x\geq A+1} + \ind{x \leq A+1,z\geq A+1}\int_0^{A/z}\beta(\dd r)\Big]
u^i_s(\dd x,\dd z) \dd s  <\infty \label{asww1}
\end{gather}
and for all $t\geq 0$, all $\varphi \in C^1_c(\cJ)$,
\begin{align}
\int_\cJ &  \varphi(x, z)u_t^i(\dd x, \dd z) = \int_\cJ \!\varphi(x,z)u_0^i(\dd x, \dd z) \notag
\\ 
&+ \intot \hskip-5pt \int_\cJ \!
\Big[\f{\p\varphi(x,z)}{\p x} + g(z) \f{\p\varphi(x,z)}{\p z}+ 
d(x,z)\! \int_0^1( \varphi(0,rz)-\varphi(x,z))\beta(\dd r)  \Big] u_s^i(\dd x, \dd z)\dd s.  \label{asww2}
\end{align}
Moreover, for all $t\geq 0$, we have 
$$
{\mathcal T}_\varrho (u^1_t, u^2_t) \leq {\mathcal T}_\varrho(u^1_0, u^2_0).
$$
\label{th:tcASstruct}
\end{theorem}

Here again, \eqref{asww2} makes sense thanks to \eqref{asww1}: if $\varphi \in C_c(\cJ)$
is supported in $[0,A]^2$, then
\begin{align*}
d(x,z)\Big|\int_0^1(\varphi(0,rz)-\varphi(x,z))\beta(\dd r)  \Big|
\leq& \ind{x\leq A+1,z\leq A+1} 2 \|\varphi\|_\infty \sup_{[0,A+]^2} d + \ind{x\ge A+1}\|\varphi\|_\infty d(x,z) \\
&+ \ind{x\leq A+1,z\geq A+1} \|\varphi\|_\infty d(x,z)\int_0^{A/z} \beta(\dd r).
\end{align*}

\vip

\begin{proof} As before, we consider $u^1_0,u^2_0 \in \cP(\cJ)$ and a coupling $v_0 \in {\mathcal H}(u^1_0,u^2_0)$.
There exists a  family $(v_t)_{t\geq 0}$ of probability measures on~$\cJ^2$ such that, for all $t\geq 0$,
\begin{align}
\intot \int_\cJ  \Big(&d(x,z)\Big[ \ind{x\geq A+1} + \ind{x \leq A+1,z\geq A+1}\int_0^{A/z}\beta(\dd r)\Big] 
\notag\\
&+ d(\tx,\tz)\Big[ \ind{\tx\geq A+1} + \ind{\tx \leq A+1,\tz\geq A+1}\int_0^{A/\tz}\beta(\dd r)\Big]\Big)
v_s(\dd x,\dd z, \dd \tx, \dd \tz) \dd s  <\infty \label{ett1}
\end{align}
and that weakly solves 
\begin{align*}
\dis \f{\p}{\p t} v_t(x,z,\wt x ,\wt z) &+ \f{\p v_t}{\p x}+ \f{\p v_t}{\p \wt x} + \f{\p }{\p z} \big[g(z)\, v_t\big] + \f{\p v_t}{\p \wt z} \big[g(\wt z)\, v_t \big] +b(x,z)\vee d(\wt x, \wt z) \,v_t 
\\
&= \delta(x) \delta(\wt x)  \intt  (d(x', z')\wedge d(\wt x', \wt z'))  
 \frac{1}{z'  \wt z'} \beta \big( \frac{z}{ z'}\big)   \delta\big(\frac{z}{ z'}-  \frac{\wt z}{\wt z'}\big)    \, v_t(\dd x',\dd z', \dd \wt x', \dd \wt z')  
\\
&\quad + \delta(x)  \int  (d(x', z') -d(\wt x, \wt z))_+ \frac{1}{ z'} \beta \big( \frac{z}{ z'}\big)   \, v_t(\dd x', \dd z',  \wt x,  \wt z) 
\\
&\quad + \delta( \wt x) \int  (d( \wt x', \wt z') -d(x,z))_+ \frac{1}{ \wt z'} \beta \big( \frac{\wt z}{\wt z'}\big)   \, v_t(x,z,\dd \wt x',  \dd \wt z')
\end{align*}
This equation means that for all $t\geq 0$, for all $\vp \in C^1_c(\cJ^2)$, 
\begin{align}
\intt \vp&(x,z,\wt x, \wt z)  v_t(\dd x,\dd z,d \wt x ,d \wt z) = \intt \vp(x,z,\wt x, \wt z) 
v_0(\dd x,\dd z,d \wt x ,d \wt z)\notag\\
&+ \intot \hskip-5pt \intt \Big[\f{\p\varphi}{\p x} +\f{\p\varphi}{\p \wt x} + g(z) \f{\p\varphi}{\p z}
+ g(\wt z) \f{\p\varphi}{\p \wt z}\Big](x,z,\wt x, \wt z)\, v_s(\dd x,\dd z, d\wt x,d\wt z)  \dd s\notag \\
&+  \intot \hskip-5pt\intt  \Big[(d(x, z)\wedge d(\wt x, \wt z))  \int_0^1 [\vp (0, r z, 0,r \wt z )-\vp(x,z,\wt x, \wt z)] 
\, \beta(\dd r)\notag \\
& \hskip2cm +(d(x, z) -d(\wt x, \wt z))_+  \int_0^1 [\vp (0, r z, \wt x , \wt z )-\vp(x,z,\wt x, \wt z)] \, \beta(\dd r)
\label{ett2}\\
& \hskip2cm +(d( \wt x, \wt z) -d(x,z))_+ \int_0^1 [\vp (x, z, 0 , r \wt z) -\vp(x,z,\wt x, \wt z)]\, \beta(\dd r)
\Big]  \, v_s(\dd x,\dd z,d \wt x ,d \wt z) \dd s. \notag
\end{align}
The two marginals of $v_t$ solve \eqref{asww1}-\eqref{asww2}. Next,
one finds as usual that there is a function $h:[0,\infty)\to[0,\infty)$,
strictly increasing to infinity, such that 
\begin{align*}
\intt &[h(x)+h(z)+h(\tx)+h(\tz)]v_t(\dd x,\dd z,\dd \tx,\dd \tz) \\
&+ \intot \intt [d(x,z)H(x,z)+d(\tx,\tz)H(\tx,\tz)] v_s(\dd x,\dd z,\dd \tx,\dd \tz)\dd s
\leq C + 2(1+\overline{C}) t,
\end{align*}
where $H(x,z)=h(x)+ \int_0^1 (h(x)-h(rx))\beta(\dd r)$ and
$\overline{C}= \sup_{x \geq 0} g(x) h'(x)$. This {\it a priori} estimate is sufficient to construct a solution to 
\eqref{ett1}-\eqref{ett2}, because for any $A\geq 1$
$$
\ind{x\geq A+1} + \ind{x \leq A+1,z\geq A+1}\int_0^{A/z}\beta(\dd r)
\leq \frac{h(x)}{h(A+1)} + \int_0^1 \frac{h(x)-h(rx)}{h(A+1)-h(A)} \beta(\dd r) \leq 
\frac{H(x)}{h(A+1)-h(A)}.
$$

Following the usual procedure to prove the decay property, that we do not repeat again, it suffices,
to conclude the proof, to check that
\begin{align*}
(d(x,z)&\vee d(\wt x, \wt z)) \min( a, | x- \wt x|+|z- \tilde z| )
\\& \geq   (d(x, z)\wedge d(\wt x, \wt z)) \int_0^1  \min(a,r |z-\wt z|) \beta(\dd r) + a \,  | d(x, z) -d(\wt x, \wt z)|  .
\end{align*}
When the $\min$ on the left hand side is achieved by $a$ the inequality is obvious. Otherwise, we have to check that 
\begin{align*}
(d(x,z)\vee d(\wt x, \wt z)) \big[ | x- \wt x|+|z- \tilde z| \big] 
\geq \int_0^1  r  \beta(\dd r)  \,   (d(x, z)\wedge d(\wt x, \wt z)) |z-\wt z| + a \,  | d(x, z) -d(\wt x, \wt z)|  .
\end{align*}
This again is satisfied if 
\begin{align*}
\Big(1-  \int_0^1  r  \beta(\dd r) \Big) (d(x,z)\vee d(\wt x, \wt z)) \big[ | x- \wt x|+|z- \tilde z| \big] 
\geq a \,  | d(x, z) -d(\wt x, \wt z)|  ,
\end{align*}
which is the condition \eqref{tc:ASstructas1Z}, recall that we are in the case where  $| x- \wt x|+|z- \tilde z|\leq a$.
\end{proof} \qed

\section{Sexually structured populations}

Here  a female of type $x'$ mates with a male of type $x'_*$, chosen with  the probability $u_t(\cdot)$, 
the newborn is distributed with type $x$ according to the law $K(x; x',x'_*)$.  As often in this theory, 
we assume the distribution of males and females are identical, and rely on the formalism which can be 
found in~\cite{burger, MirRao, CGP2019} for instance. 
The (homogeneous)  model reads
\begin{equation}
\p_t u_t( x) +u_t( x) =  \intt_{\R^{2d}} K(x; x', x'_*) u_t( \dd x') u_t( \dd x'_*),
\qquad t \geq 0,\;x \in \R^d.
\label{eq:resgeneralform}
\end{equation}
For keeping the total population constant, the  kernel $K\geq 0$  satisfies 
$$
 \intt_{\R^d} K( \dd x ; x', x'_*) =1.
 $$
For instance, we can think of two extreme cases of either a Dirac concentration or a uniform distribution, 
\beq
K(x; x',x'_*) =\delta_{\theta x' +(1-\theta) x'_*}(x), \quad   \theta \in (0,1),  \quad   \text{or} \quad K(x; x',x'_*) = \f{1}{|x' - x'_*|} \ind{x \in (x',x'_*)}.
\label{tc:adsKex} 
\eeq
These distributions can be generalized to the form
\beq
K(\dd x; x',x'_*) = \int_{0}^1  \delta_{ x' \sigma+ x'_*(1-\sigma)} (x) h(\dd \sigma),
\label{tc:adsKex3} 
\eeq
with  $h$ a probability distribution on $[0,1]$ such that $\int_0^1  \sigma h(\dd \sigma) = \theta \in (0,1)$, 
which is the form we use in the sequel. 

\begin{theorem} With the notations and assumptions above, we choose, for some $p\geq 1$,
$$
\varrho(x,y) =  | x-y|^p. 
$$
For any $u_0^1,u_0^2 \in \cP_p(\R^d)$, there exists a pair of weak measure solutions
$(u_t^1)_{t\geq 0},(u_t^2)_{t\geq 0} \subset \cP_p(\R^d)$ to~\eqref{eq:resgeneralform}, starting from $u_0^1$ and
$u_0^2$, i.e., such that for $i=1,2$, all $t\geq 0$ and all $\varphi \in C^1_c(\R^d)$,
\begin{align}
\int_{\R^d}  & \varphi(x)u_t^i(\dd x) = \int_{\R^d} \!\varphi(x)u_0^i(\dd x) \notag \\
&+ \intot \hskip-5pt \inttt
\big[  \varphi(x)  - \theta \varphi(x')- (1-\theta) \varphi(x'_*) \big] K(\dd x; x',x'_*)  u_s^i(\dd x', \dd x'_*)\dd s.  \label{showw2}
\end{align}
Moreover, for all $t\geq 0$, we have 
$$
{\mathcal T}_\varrho (u_t^1, u_t^2) \leq {\mathcal T}_\varrho(u_0^1, u_0^2).
$$
\label{th:tcsex1}
\end{theorem}

\begin{proof}
We use a coupling $ \overline K(x,y; x',x'_*,y', y'_*)$, to be chosen later, with the property
\[
 \int_{\R^d}  \overline K(x,\dd y; x',x'_*,y', y'_*) =K(x; x', x'_*) , \qquad  \int_{\R^d}  \overline K(\dd x,y; x',x'_*,y', y'_*)  = K(y; y', y'_*).
\]
Then, we introduce  the coupling equation
\[
  \p_t v_t( x,y) +v_t(x,y)= \intttt \overline K(x,y; x',x'_*,y', y'_*) v_t( \dd x',\dd y') v_t(\dd x'_*, \dd y'_*)  .
\]
This means that for all $t\geq 0$, for all $\vp \in C^1_c(\R^{2d})$,  
\beq \begin{aligned}
 \int_{\R^d} \vp(x,y) v_t(\dd x,\dd y)   =  \int_{\R^d} \vp(x,y) v_0(\dd x,\dd y) +& \intot  \hskip-5pt  \intttt  \big[ \vp(x,y)  -\theta \vp(x',y') - (1- \theta) \vp(x'_*,y'_*)  \big]  
 \\[5pt]
 & \overline K(\dd x ,\dd y; x',x'_*,y', y'_*) v_s( \dd x',\dd y') v_s(\dd x'_*, \dd y'_*) \dd s,
\end{aligned} 
\label{sspv}
\eeq
where $\theta \in (0,1)$ is the same as in \eqref{tc:adsKex3}. For existence, we have to check the 
tightness in $ \cP_p(\R^d)$. By the de la Vall\'ee Poussin theorem, there is a function 
$h:\R^d \to [0,\infty)$ such that
$\lim_{|x| \to \infty} h(x)=\infty$ and
$$
\overline C:=\intt [h(x)|x|^p+h(y)|y|^p] v_0(\dd x, \dd y)  <\infty.
$$
One can moreover choose $h$ smooth, satisfying $h(0)=0$ and such that
$x\mapsto h(x)|x|^p$ convex.  Choosing $h(x)|x|^p+h(y)|y|^p$ as a test function in \eqref{sspv}, 
we conclude the bound, for all $t\geq 0$,
\[
\intt [h(x)|x|^p+h(y)|y|^p] v_t(\dd x, \dd y)  \leq \overline C,
\]
because, by convexity, the second term in the right hand side of~\eqref{sspv} is non-positive.
\\

For the non-expansion property, we just have to show that the right hand side is non-positive 
for $\varrho(x,y) =  | x-y|^p $ (arguing again after truncation, regularization), if choosing
as coupling kernel 
\[
\overline K(\dd x ,\dd y; x',x'_*,y', y'_*)=  \int_{0}^1 h(\dd \sigma) 
\delta_{ \sg x'+ (1-\sg) x'_*} (x) \delta_{ \sg y'+ (1-\sg) y'_* } (y).
\]
The duality formula raises the condition
 \[
 \int_{0}^1  \big | \sigma x'+(1-\sigma) x'_* - \sigma   y'- (1-\sigma) y'_* \big |^p h(\dd \sigma)  \leq  \theta  |x'-y'|^p +  (1-\theta)  |x'_* -y'_* |^p
 \]
which, by convexity, is immediate.

\end{proof} \qed

\appendix
\section{Uniqueness of measure solutions}
The coupling method is most powerful when  the measure solutions are unique. This uniqueness problem, 
in particular for coefficients with low regularity, can lead to several deep developments, 
\cite{BCG2020, DDGW2019}. Here, we consider regular coefficients so that the Hilbert Uniqueness 
Method can be applied without difficulty both to the Structured Equations under consideration and to the coupled 
equations. We treat in details the example of the renewal equation, i.e.,~\eqref{renewalEq} when $b=\delta$.
We assume that
\beq
d \in C([0,\infty)), \qquad g(x) \in C^1_{b} (\R^+), \qquad g(0)\geq 0,
\label{as:rengen1}
\eeq

We define the weak solutions (or distributional solutions), as follows.

\begin{definition} A function $(u_t)_{t\geq 0} \subset  \cP(0,\infty)$ satisfies the renewal equation~\eqref{renewalEq} in the distribution sense,  if for all $T>0$ and all {\em test function} $\psi \in C^1_{\rm comp} \big([0,T] \times[0, \infty[\big)$ such that $\psi(x,T)\equiv 0$, we have
\[
\dis- \int_0^T  \hskip-5pt  \int_0^\infty \dis \left [  \f{\p \psi(x,t)}{\p t}  +g(x) \f{\p \psi(x,t)}{\p x}  -  d(x) \psi(x,t) + d(x)  \psi(0,t)  \right ] u_t(\dd x) \; \dd t 
=  \dis \int_0^\infty \psi(x,0) u_{0}(\dd x) .
\]
\label{def:rendis}
\end{definition}

\begin{theorem} [Well posednesss] We assume \eqref{as:rengen1}. There is a unique weak solution of 
the renewal equation~\eqref{renewalEq}.
\label{th:renexis}
\end{theorem}

For the existence part, we refer to \cite{BCG2020, DDGW2019} where more elaborate equations are treated. For uniqueness, need to study the inhomogeneous dual problem. We introduce a  source term $S(x,t)$ on a given time interval $[0,T]$ and
\begin{equation} \left\{
	\begin{array}{l}
- \f{\p}{\p t} \psi(x,t) - g(x) \f{\p}{\p x} \psi(x,t) + d(x) \, \psi(x,t) = \psi(0,t) d(x) + S(x,t), 
\\[5pt]
\psi( x,T)=0.
\end{array} \right.
\label{eq:renewalad}
\end{equation}
This problem is backward in $t$ and $x$, therefore it does not use a boundary condition at $x=0$.

\begin{lemma} [Existence for the dual problem]Assume \eqref{as:rengen1}, $S\in C^1_{\rm comp}\big([0,T) \times \R^+\big)$ and $d\in C^1(0,\infty)$, then there is a unique $C^1$ solution to the dual  equation \eqref{eq:renewalad}. Moreover $\psi(x,t)$ vanishes for  $x\geq R >0$ for some $R$ depending on the data and $T$, and the bound holds
$$
\dis \sup_{0\leq t \leq T, \, x \in \R^+} |\psi (x,t)| \leq C(T) \| S\|_\infty.
$$
\label{lm:renadjuni}
\end{lemma}

\begin{proof}
We use the method of characteristics based on the solution of the differential system parametrized by the Cauchy data $(x,t)$ which is fixed
$$\left\{
	\begin{array}{l}
\f{d}{ds} X_s = g(X_s), \quad 0 \leq s \leq T, 
\\[5pt]
X_t=x \geq 0.
\end{array} \right.
$$
It is well-posed  thanks to the Cauchy-Lipschitz theorem and $X_s \geq 0$ thanks to assumption $g(0) \geq 0$.  It might be useful to keep in mind that $X_s$ depends on $(x,t)$ and thus the notation $X_s \equiv X_s(x,t)$.
\\

Then, we set
$$
\wt \psi(s; x,t) = \psi(s,X_s) e^{\int_s^t d(\sg,X_\sg) \dd \sg}, \quad \wt d(s; x,t) =d(s,X_s) e^{\int_s^t d(\sg,X_\sg) \dd \sg} ,
$$
$$
 \wt S(s; x,t) =S(s,X_s) e^{\int_s^t d(\sg,X_\sg) \dd \sg} ,
$$
and ignore the parameter $(x,t)$ when the statements are clear enough.  We rewrite equation \eqref{eq:renewalad} as 
$$ \begin{array}{rl}
\f{d}{d s}\wt \psi(s) & = \big[ \f{\p}{\p t} \psi + g \f{\p}{\p x} \psi - d \, \psi \big] e^{\int_s^t d(\sg,X_\sg) \dd \sg}\Big|_{(s, X_s)} 
 \\ [10pt]
&= - \psi(0,s) \wt d(s) - \wt S(s),
\end{array} 
$$
Next, we integrate between $s=t$ and $s=T$, use the Cauchy data at $t=T$ and the identity $\wt \psi(t)=\psi(x,t)$, and we obtain
\beq 
  \psi(x,t) = \int_t^{T} \big[\psi( 0,s) \wt d(s;x,t) + \wt S(s;x,t) \big] \dd s.
\label{eq:renadj1}
\eeq
This integral equation can be solved first for $x=0$. Then, equation \eqref{eq:renadj1} is reduced to the  Volterra equation
$$
 \psi(0,t) = \int_t^{T} \big[ \psi(0, s) \wt d(s;0,t) + \wt S(s;0,t) \big] \dd s , \qquad 0\leq t \leq T
$$
which, thanks to the (backward) Cauchy-Lipschitz theorem,  has a unique solution   that vanishes for $t=T$.  By the $C^1$ regularity of the data, we also have $\psi(0,t)\in C^1([0,T])$.
\\

Since $\psi(0,t)$ is now known, formula \eqref{eq:renadj1} gives us the explicit form of the solution for all $(x,t)$. 
Notice that, in the compact support statement, $\wt \psi(x,t)$ vanishes for $x\geq R$ where $R$ denotes the size of the support of   $S$ in $x$, plus $T \|g\|_\infty$.  The uniform bound on $\psi$ also follows from formula~\eqref{eq:renadj1},

\end{proof} \qed

\bigskip

\begin{proof}{\bf [Uniqueness for the renewal equation.]}
With the help of the dual problem, we can use the Hilbert Uniqueness Method. The idea is simple: when the coefficients $d$, $g$ satisfy the assumptions of Lemma \ref{lm:renadjuni}, we can use the solution $\psi$ of \eqref{eq:renewalad} as a test function in the weak formulation of Definition~\ref{def:rendis}. For the difference $u=u^2-u^1$ between two possible solutions $u^2, \, u^1$ with the same initial data, we arrive at 
$$
 \dis \int_0^T  \hskip-5pt  \int_0^\infty  \left [  \f{\p \psi(x,t)}{\p t}  +g(x) \f{\p \psi(x,t)}{\p x}  -  d(x) \psi(x,t) + d(x) \psi(0,t) \right ] u_t(\dd x)  \dd t  =0.
$$
for $\psi(\cdot,\cdot)\in C^1$ which is the case when $d \in C^1$. Then, taking into account~\eqref{eq:renewalad},  we arrive at 
$$
 \dis \int_0^T  \hskip-5pt  \int_0^\infty  S(x,t) u_t(\dd x)  \dd t =0,
$$
for all $T>0$ and  all  functions $S\in C^1_{\rm comp}$, and this implies $u\equiv 0$.
\\

When $d$  is merely continuous, we consider a regularized family 
$d_p \to d$ where the convergence holds locally  uniformly. Then, for a given function $S\in C^1_{\rm comp}$, we solve \eqref{eq:renewalad} with $d_p$ in place of $d$ and call $\psi_p$ its solution (which is uniformly  bounded with compact support). Inserting it in the definition of weak solutions, we obtain
$$
 \dis \int_0^T  \hskip-5pt  \int_0^\infty S(x,t) u_t(\dd x) \dd t=  R_p,
$$ 
$$
 R_p =  \dis \int_0^T  \hskip-5pt  \int_0^\infty   [d_p- d(x)] [\psi_p(x,t) -\psi_p(0,t) ] u_t(\dd x)  \; \dd t ,
$$
and using that  $\psi_p$ is uniformly bounded, we deduce that 
$$
|R_p| \leq T\, \| \psi\|_\infty\,  \|d_p- d \|_\infty  {\;}_{\overrightarrow{\; p \rightarrow \infty \; }}\; 0.
$$
Therefore, we have recovered the identity $\dis \int_0^T \hskip-5pt \int_0^\infty S(x,t) u_t(\dd x) \dd t =0,$ for all functions $S\in C^1_{\rm comp}$, and this implies again $u\equiv 0$.
\\

This concludes the uniqueness result stated in Theorem \ref{th:renexis}.
\end{proof} \qed 

%
%
%

\bibliographystyle{siam}  
\bibliography{BibrefW}

\begin{thebibliography}{10}

\bibitem{AGSbook}
{\sc L.~Ambrosio, N.~Gigli, and G.~Savar\'{e}}, {\em Gradient flows in metric
  spaces and in the space of probability measures}, Lectures in Mathematics ETH
  Z\"{u}rich, Birkh\"{a}user Verlag, Basel, second~ed., 2008.

\bibitem{BCG2020}
{\sc V.~Bansaye, B.~Cloez, and P.~Gabriel}, {\em Ergodic behavior of
  non-conservative semigroups via generalized {D}oeblin's conditions}, Acta
  Appl. Math., 166 (2020), pp.~29--72.

\bibitem{BerryLG2016}
{\sc H.~Berry, T.~Lepoutre, and A.~M. Gonz\'{a}lez}, {\em Quantitative
  convergence towards a self-similar profile in an age-structured renewal
  equation for subdiffusion}, Acta Appl. Math., 145 (2016), pp.~15--45.

\bibitem{BertoinW2020}
{\sc J.~Bertoin and A.~R. Watson}, {\em The strong {M}althusian behavior of
  growth-fragmentation processes}, Ann. H. Lebesgue, 3 (2020), pp.~795--823.

\bibitem{bianchiniG2011}
{\sc S.~Bianchini and M.~Gloyer}, {\em An estimate on the flow generated by
  monotone operators}, Comm. Partial Differential Equations, 36 (2011),
  pp.~777--796.

\bibitem{BJM2005}
{\sc F.~Bouchut, F.~James, and S.~Mancini}, {\em Uniqueness and weak stability
  for multi-dimensional transport equations with one-sided {L}ipschitz
  coefficient}, Ann. Sc. Norm. Super. Pisa Cl. Sci. (5), 4 (2005), pp.~1--25.

\bibitem{burger}
{\sc R.~B{\"u}rger}, {\em {The mathematical theory of selection, recombination,
  and mutation}}, {Wiley Series in Mathematical and Computational Biology},
  John Wiley \& Sons, Ltd., Chichester, 2000.

\bibitem{CGM2019}
{\sc V.~Calvez, P.~Gabriel, and A.~Mateos~Gonz\'{a}lez}, {\em Limiting
  {H}amilton-{J}acobi equation for the large scale asymptotics of a
  subdiffusion jump-renewal equation}, Asymptot. Anal., 115 (2019), pp.~63--94.

\bibitem{CGP2019}
{\sc V.~Calvez, J.~Garnier, and F.~Patout}, {\em Asymptotic analysis of a
  quantitative genetics model with nonlinear integral operator}, J. \'{E}c.
  polytech. Math., 6 (2019), pp.~537--579.

\bibitem{CCDR}
{\sc J.~Chevallier, M.~J. C\'{a}ceres, M.~Doumic, and P.~Reynaud-Bouret}, {\em
  Microscopic approach of a time elapsed neural model}, Math. Models Methods
  Appl. Sci., 25 (2015), pp.~2669--2719.

\bibitem{CushingBook}
{\sc J.~M. Cushing}, {\em An introduction to structured population dynamics},
  vol.~71 of CBMS-NSF Regional Conference Series in Applied Mathematics,
  Society for Industrial and Applied Mathematics (SIAM), Philadelphia, PA,
  1998.

\bibitem{DDGW2019}
{\sc T.~Debiec, M.~Doumic, P.~Gwiazda, and E.~Wiedemann}, {\em Relative entropy
  method for measure solutions of the growth-fragmentation equation}, SIAM J.
  Math. Anal., 50 (2018), pp.~5811--5824.

\bibitem{dpl}
{\sc R.~J. DiPerna and P.-L. Lions}, {\em Ordinary differential equations,
  transport theory and {S}obolev spaces}, Invent. Math., 98 (1989),
  pp.~511--547.

\bibitem{Dob}
{\sc R.~L. Dobru\v{s}in}, {\em Vlasov equations}, Funktsional. Anal. i
  Prilozhen., 13 (1979), pp.~48--58, 96.

\bibitem{DHKR}
{\sc M.~Doumic, M.~Hoffmann, N.~Krell, and L.~Robert}, {\em Statistical
  estimation of a growth-fragmentation model observed on a genealogical tree},
  Bernoulli, 21 (2015), pp.~1760--1799.

\bibitem{DOR2020}
{\sc M.~Doumic, A.~Olivier, and L.~Robert}, {\em Estimating the division rate
  from indirect measurements of single cells}, Discrete Contin. Dyn. Syst. Ser.
  B, 25 (2020), pp.~3931--3961.

\bibitem{feller}
{\sc W.~Feller}, {\em On the integral equation of renewal theory}, Ann. Math.
  Statistics, 12 (1941), pp.~243--267.

\bibitem{ferretti}
{\sc L.~Ferretti, C.~Wymant, M.~Kendall, L.~Zhao, A.~Nurtay,
  L.~Abeler-D{\"o}rner, M.~Parker, D.~Bonsall, and C.~Fraser}, {\em Quantifying
  sars-cov-2 transmission suggests epidemic control with digital contact
  tracing}, Science, 368 (2020).

\bibitem{FL2016}
{\sc N.~Fournier and E.~L\"{o}cherbach}, {\em On a toy model of interacting
  neurons}, Ann. Inst. Henri Poincar\'{e} Probab. Stat., 52 (2016),
  pp.~1844--1876.

\bibitem{FoPe2021}
{\sc N.~Fournier and B.~Perthame}, {\em Transport distances for pdes: the
  coupling method}, EMS Surv. Math. Sci., 7 (2020), pp.~1--31.

\bibitem{GoMoPa}
{\sc F.~Golse, C.~Mouhot, and T.~Paul}, {\em On the mean field and classical
  limits of quantum mechanics}, Comm. Math. Phys., 343 (2016), pp.~165--205.

\bibitem{Ha}
{\sc M.~Hauray}, {\em Wasserstein distances for vortices approximation of
  {E}uler-type equations}, Math. Models Methods Appl. Sci., 19 (2009),
  pp.~1357--1384.

\bibitem{MaPu}
{\sc C.~Marchioro and M.~Pulvirenti}, {\em Mathematical theory of
  incompressible nonviscous fluids}, vol.~96 of Applied Mathematical Sciences,
  Springer-Verlag, New York, 1994.

\bibitem{MD_LN68}
{\sc J.~A.~J. Metz and O.~Diekmann}, {\em The dynamics of physiologically
  structured populations}, vol.~68 of Lecture Notes in Biomath., Springer,
  Berlin, 1986.

\bibitem{MirRao}
{\sc S.~Mirrahimi and G.~Raoul}, {\em {Dynamics of sexual populations
  structured by a space variable and a phenotypical trait}}, {Theoretical
  Population Biology}, {84} ({2013}), pp.~{87--103}.

\bibitem{MW2018}
{\sc S.~Mischler and Q.~Weng}, {\em Relaxation in time elapsed neuron network
  models in the weak connectivity regime}, Acta Appl. Math., 157 (2018),
  pp.~45--74.

\bibitem{Monmarche2015}
{\sc P.~Monmarch{\'e}}, {\em On {$\mathcal H^1$} and entropic convergence for
  contractive {PDMP}}, Electron. J. Probab., 20 (2015), pp.~Paper No. 128, 30.

\bibitem{NPT2018}
{\sc S.~Nordmann, B.~Perthame, and C.~Taing}, {\em Dynamics of concentration in
  a population model structured by age and a phenotypical trait}, Acta Appl.
  Math., 155 (2018), pp.~197--225.

\bibitem{PPS2010}
{\sc K.~Pakdaman, B.~Perthame, and D.~Salort}, {\em Dynamics of a structured
  neuron population}, Nonlinearity, 23 (2010), pp.~55--75.

\bibitem{P_birkhauser}
{\sc B.~Perthame}, {\em Transport equations in biology}, Frontiers in
  Mathematics, Birkh\"{a}user Verlag, Basel, 2007.

\bibitem{Santambrogio}
{\sc F.~Santambrogio}, {\em Optimal transport for applied mathematicians},
  vol.~87 of Progress in Nonlinear Differential Equations and their
  Applications, Birkh\"{a}user/Springer, Cham, 2015.

\bibitem{VTOT}
{\sc C.~Villani}, {\em Topics in optimal transportation}, vol.~58 of Graduate
  Studies in Mathematics, American Mathematical Society, Providence, RI, 2003.

\end{thebibliography}

\end{document}